\documentclass[bibliography=totoc, 12pt]{amsart}
\usepackage[utf8]{inputenc}
\usepackage[margin=1in]{geometry} 
\usepackage{amsmath,amsthm,amssymb,amsfonts, dsfont, graphicx,mathtools,mathrsfs,tikz,hyperref,physics, stmaryrd, bm, tikz-cd,tkz-euclide,comment, enumerate, appendix, todonotes}

\usepackage{subcaption}
\usepackage[backend=biber,bibencoding=utf8,style=numeric,url=false,
    isbn=false,     
    doi=false,           
    abbreviate=true,            
    giveninits=true, 
    natbib=true,
    maxbibnames=9,
    hyperref=true]{biblatex}
\usetikzlibrary{decorations.pathreplacing,decorations.markings}

\def\O{\mathcal{O}}

\def\R{\mathbb{R}}
\def\Z{\mathbb{Z}}

\def\Q{\mathbb{Q}}

\def\tr{\mathrm{tr}}

\theoremstyle{plain}
\newtheorem{theorem}{Theorem}[section] 

\newtheorem{lemma}[theorem]{Lemma}

\numberwithin{equation}{section}
\theoremstyle{definition}

\theoremstyle{remark}

\title[Cubic forms over $\mathbb{Q}(\sqrt{-D})$ and pairs of rational cubic forms]{Cubic forms over imaginary quadratic number fields and pairs of rational cubic forms}
\author[C. Bernert]{Christian Bernert}
\author[L. Hochfilzer]{Leonhard Hochfilzer}
\address{Mathematisches Institut, Bunsenstraße 3-5, 37073 Göttingen, Germany}
\email{christian.bernert@uni-goettingen.de,  leonhard.hochfilzer@uni-goettingen.de}

\begin{document}
\maketitle
\begin{abstract}
    We show that every cubic form with coefficients in an imaginary quadratic number field $K/\mathbb{Q}$ in at least $14$ variables represents zero non-trivially. This builds on the corresponding seminal result by Heath-Brown for rational cubic forms. As an application we deduce that a pair of rational cubic forms has a non-trivial rational solution provided that $s \geq 627$. Furthermore, we show that every rational cubic hypersurface in at least $33$ variables contains a rational line, and that every rational cubic form in at least $33$ variables has "almost-prime" solutions.
\end{abstract}

\section{Introduction} \label{sec.intro_cubes}
The study of integer solutions to polynomial equations is one of the most fundamental mathematical problems. Quadratic forms are very well understood but the situation already becomes much more difficult when studying cubic equations. A cubic form $C(\mathbf{x}) \in \Z[x_1, \hdots, x_s]$ is a homogeneous  polynomial of degree $3$. We say that $C$ represents zero non-trivially if there is a vector $\mathbf{x} \in \mathbb{Z}^s \backslash \{\mathbf{0}\}$ such that $C(\mathbf{x})=0$.
Lewis~\cite{lewis_57_existence} and Birch~\cite{birch57existence} both independently showed that every cubic form in sufficiently many variables represents zero non-trivially.
 Using the Hardy-Littlewood circle method, Davenport~\cite{davenport_32_variables} showed that it suffices to assume  $s \geq 32$ in order to show that $C$ represents zero non-trivially, which he then improved to $s \geq 16$ in a series of papers~\cite{davenport_29_vars, davenport63}. The current state of the art is due to Heath-Brown~\cite{heath2007cubic} who showed that $14$ variables suffice.
 
The best one can hope for is that every cubic forms in at least $10$ variables represents zero non-trivially,  since there exist cubic forms in $9$ variables, which do not have non-trivial $p$-adic solutions and hence also do not represent zero non-trivially over the integers. 

More is known when the cubic form is assumed to be non-singular. In this case Heath-Brown~\cite{heath_brown_ten} showed that if $s\geq 10$ then the cubic form represents zero non-trivially, and Hooley~\cite{Hooley+1988+32+98} established the Hasse Principle if $s \geq 9$. That is, he showed that if a non-singular cubic form over $\Q$ in at least nine variables has a non-trivial $p$-adic solution for every $p$ and a non-trivial real solution then it also represents zero non-trivially over the rational numbers.

One may also consider these problems for cubic forms over a number field $K/\Q$. 
In fact the above mentioned result by Lewis was proved for any number field $K/\Q$. Using the circle method the number of variables required was reduced to $54$ by Ramanujam~\cite{ramanujam1963cubic}, which was subsequently improved to $17$ variables by Ryavec~\cite{ryavec1969cubic} and $16$ variables by Pleasants~\cite{pleasants1975cubic}. If one assumes the cubic form to be non-singular then recent work by Browning--Vishe~\cite{browning_vishe_ten} shows that ten variables suffice in order to infer the existence of a non-trivial zero, which improves previous work by Skinner~\cite{skinner_94_thirteen_vars}.
The main result of this paper is the following.
\begin{theorem} \label{thm.cubic_in_14}
Let $K/\mathbb{Q}$ be an imaginary quadratic number field. If $C(\mathbf{x})$ is a homogeneous cubic form over $K$ in at least $14$ variables then $C(\mathbf{x})$ represents zero nontrivially.
\end{theorem}

It seems likely that our result should remain true for general number fields, however there are two serious obstructions in generalizing Heath-Brown's ideas to the number field setting, as we discuss in the course of our proof. We are able to remove these difficulties only in the special case of imaginary quadratic number fields.

Our result has some interesting applications to problems that do not involve, prima facie, any number fields. In particular, we are able to significantly reduce the number of variables needed in order to solve a pair of arbitrary cubic forms.

\begin{theorem}\label{thm.pairs}
Let $C_1, C_2 \in \mathbb{Z}[x_1, \hdots, x_s]$ be two cubic forms. If $s \geq 627$ then there exists a non-trivial integer solution to the system $C_1(\mathbf{x}) = C_2(\mathbf{x}) = 0$. 
\end{theorem}

Previously it was shown by Dietmann and Wooley~\cite{dietmannwooley2003} that $828$ variables suffice to solve a pair of cubic forms. If the variety defined by $C_1=C_2=0$ is assumed to be non-singular, work of Brandes~\cite{brandes_note_p_adic} building on work of Dumke~\cite{dumke2014quartic} shows that $132$ variables suffice. As in the work of Dietmann--Wooley, our result follows by constructing rational linear subspaces on a rational cubic hypersurface. Theorem~\ref{thm.cubic_in_14} enables us to prove the existence of rational linear subspaces of a given dimension by requiring fewer variables than in previous work.

\begin{theorem} \label{thm.lines}
Let $C \in \mathbb{Z}[x_1, \hdots, x_s]$ be a cubic form and let $m \ge 0$ be an integer. If \[s>\frac{5m^2+29m}{2}+\begin{cases} 13 & 2 \mid m\\15 & 2 \nmid m \end{cases},\]
then the projective hypersurface $X$ defined by $C(\mathbf{x}) = 0$ contains a rational projective linear space of dimension $m$. In particular, if $s \geq 33$
then $X$ contains a rational projective line.
\end{theorem}
This improves on the work of Dietmann--Wooley who required
\[s>\frac{5m^2+33m}{2}+\begin{cases} 15 & 2 \mid m\\17 & 2 \nmid m \end{cases}.\]
On choosing $m=13$, Theorem~\ref{thm.pairs} follows immediately from Theorem~\ref{thm.lines}, as the latter allows us to find a $13$-dimensional projective linear subspace on the surface given by $C_1(\mathbf{x})=0$, on which we can solve $C_2(\mathbf{x})=0$ in view of Heath-Brown's $14$-variable result.

The improvement in Theorem~\ref{thm.pairs} thus arises from two different sources: By Heath-Brown's improvement of Davenport's 16-variable result, it suffices to choose $m=13$ instead of $m=15$ in the application of Theorem~\ref{thm.lines}. This alone suffices to establish Theorem~\ref{thm.pairs} for $s \ge 655$. On the other hand, to prove our version of Theorem~\ref{thm.lines} and hence deduce Theorem~\ref{thm.pairs} for $s \ge 627$ we really require our Theorem~\ref{thm.cubic_in_14}, i.e. the improvement of Pleasants' $16$-variable result in the number field case.

We also note that in the case when we seek to show the existence of a rational projective line, i.e. the case $m=1$ in Theorem~\ref{thm.lines}, our result improves on work of Wooley~\cite{wooley1997} who had the same result under the assumption $s \ge 37$. Another two variables can be saved using ideas from forthcoming work by Brandes and Dietmann~\cite{brandes_dietmann_II}, thus leading to a result for $s \ge 31$ variables. 

More specifically, while our argument (building on Wooley's) only requires Theorem \ref{thm.cubic_in_14} for one imaginary quadratic number field (e.g. $\mathbb{Q}(i)$), the full generality of Theorem \ref{thm.cubic_in_14} is required in the argument of Brandes and Dietmann. 

It is also worth mentioning that in a different paper of the same authors~\cite{brandes_dietmann}, the existence of a rational projective line for $s \ge 31$ variables is already established under the assumption that the underlying hypersurface is nonsingular.

Based on an observation of Brüdern--Dietmann--Liu--Wooley~\cite{birchgoldbach2010}, the existence of rational lines can be used in conjunction with the Green--Tao Theorem to produce almost prime solutions to cubic forms as follows:

\begin{theorem} \label{thm.primes}
Let $C$ be a cubic form in $s \ge 33$ variables with rational coefficients. Then there are almost prime solutions to $C(\mathbf{x})=0$ in the following sense: There are coprime integers $c_1,\dots,c_s$ such that the equation
\[C(c_1p_1,c_2p_2,\dots,c_sp_s)=0\]
has infinitely many solutions in primes $p_1,\dots,p_s$, not all equal.
\end{theorem}

We note that one can obtain the same result for $s \ge 31$, assuming the corresponding version of Theorem \ref{thm.lines}.

For comparison, the existence of prime solutions is only known for non-singular cubic forms (satisfying suitable local conditions) and under the assumption of a much larger number of variables, cf. the work of Yamagishi~\cite{yamagishi_prime} and Liu--Zhao~\cite{liu2021forms}. These authors require 8996 and 9216 variables, respectively, in the case of cubic forms, although an inspection of the proof of Lemma 8.2 in \cite{liu2021forms} shows that the argument only requires 4740 variables in this special case.

\subsection*{Notation}

We use $e(\alpha)=e^{2\pi i\alpha}$ and the notation $O(\dots)$ and $\ll$ of Landau and Vinogradov, respectively. All implied constants are allowed to depend on the number field $K$, a choice of integral basis $\Omega$ for $K$, the cubic form $C$ and a small parameter $\varepsilon >0$ whenever it appears.

As is convenient in analytic number theory, this parameter $\varepsilon$ may change its value finitely many times. In particular, we may write something like $M^{2\varepsilon} \ll M^{\varepsilon}$.

We often use the notation $q \sim R$ to denote the dyadic condition $R<q \le 2R$.

\subsection*{Acknowledgements}
We thank Tim Browning for useful discussions. We are also grateful to Julia Brandes and Rainer Dietmann for alerting us to their work~\cite{brandes_dietmann_II} and sharing a preprint. 
This work was carried out while the authors were Ph.D. students at the University of Göttingen, supported by the DFG Research Training Group 2491 \lq Fourier Analysis and Spectral Theory\rq{}.
\section{Deduction of Theorems \ref{thm.lines} and \ref{thm.primes}} \label{sec.deduction}

In this section, we give the proofs of Theorems \ref{thm.lines} and \ref{thm.primes} assuming Theorem \ref{thm.cubic_in_14}.

We begin with the observation that the existence of a $m$-dimensional rational projective linear space on the cubic hypersurface defined by $C$ is equivalent to the existence of linearly independent vectors $\mathbf{v}$ and $\mathbf{w}_1,\dots,\mathbf{w}_m$ such that $C(\mathbf{v}+t_1\mathbf{w}_1+\dots+t_m\mathbf{w}_m)=0$ identically in $t_1,\dots,t_m$. Expanding this formally as a cubic polynomial in $t_1,\dots,t_m$, we obtain
\[
C(\mathbf{v})+\sum_{i=1}^m t_i Q_i(\mathbf{v})+\sum_{1 \le i \le j \le m} t_it_jL_{i,j}(\mathbf{v})+C(t_1\mathbf{w_1}+\dots+t_m\mathbf{w_m})=0
\]
for certain quadratic and linear forms $Q_{i}$ and $L_{i,j}$ respectively, depending on $\mathbf{w}_1,\dots,\mathbf{w}_m$. We therefore need to find linearly independent $\mathbf{v}$ and $\mathbf{w}_1,\dots,\mathbf{w}_m$ such that 
\[
C(\mathbf{v})=Q_{i}(\mathbf{v})=L_{i,j}(\mathbf{v})=C(t_1\mathbf{w_1}+\dots+t_m\mathbf{w_m})=0, \quad 1 \leq i \leq j \leq m
\]
is satisfied.
By induction, we may choose linearly independent vectors $\mathbf{w}_1,\dots,\mathbf{w}_m$ satisfying of $C(t_1\mathbf{w_1}+\dots+t_m\mathbf{w_m})=0$. The linear equations $L_{i,j}(\mathbf{v})=0$ and the linear independence to $\mathbf{w}_1,\dots,\mathbf{w}_m$ then reduce the degrees of freedom for $\mathbf{v}$ by $m+\binom{m+1}{2}$. We are thus looking for a solution to the system $C(\mathbf{v})=Q_{1}(\mathbf{v})=\dots=Q_m(\mathbf{v})=0$ of one cubic and $m$ quadratic equations in $s-m-\binom{m+1}{2}$ variables. If we knew that the quadratic forms $Q_{i}$ were sufficiently indefinite, we could infer the existence of a sufficiently large linear space on which all the $Q_i$ vanish, leaving us with a single cubic form in many variables, that can be dealt with by the work of Heath-Brown~\cite{heath2007cubic}.
The crux however is that it is in general hard to control the signature of the $Q_{i}$. Instead we avoid the indefiniteness issue by passing to an imaginary quadratic number field of $\mathbb{Q}$, thus requiring our Theorem~\ref{thm.cubic_in_14}.

We now present the complete argument in order: If $m=0$, Theorem~\ref{thm.lines} is just a restatement of Heath-Brown's result on cubic forms in $14$ variables. Now suppose that $m \ge 1$. By induction, we may choose  $\mathbf{w}_1,\dots,\mathbf{w}_m \in \mathbb{Q}^s \backslash \{\bm{0}\}$ such that $C(t_1\mathbf{w}_1+\dots+t_m\mathbf{w}_m)=0$. 

Letting $K/\mathbb{Q}$ be any imaginary quadratic number field, we next show the existence of a vector $\mathbf{v} \in K^s$, linearly independent to $\mathbf{w}_1,\dots,\mathbf{w}_m$ and satisfying
\[
C(\mathbf{v})=Q_i(\mathbf{v})=L_{i,j}(\mathbf{v})=C(t_1\mathbf{w_1}+\dots+t_m\mathbf{w_m})=0.
\]
To this end, we use that a variety $Q_i(\mathbf{x})=0$, $1 \leq i \leq m$ defined by $m$ quadratic forms in $s$ variables contains a $d$-dimensional projective $K$-linear subspace, provided that 
\[s >\beta(m,d)=2m^2+d(m+1)+\begin{cases} 0 & 2 \mid m\\ 2 & 2 \nmid m \end{cases}.\]
This is equation (2.11) in \cite{dietmannwooley2003} and is easily proved by induction in view of the base cases $\beta(1,0)=4$ and $\beta(2,0)=8$ given by the Hasse--Minkowski theorem and a result of Colliot-Thélène, Sansuc and Swinnerton-Dyer~\cite{colliot}, respectively.

We thus obtain a $13$-dimensional projective linear space orthogonal to all the vectors $\mathbf{w}_i$ on which all the forms $Q_i$ and $L_{i,j}$ vanish provided that
\[s>m+\binom{m+1}{2}+\beta(m,13)=\frac{5m^2+29m}{2}+\begin{cases} 13 & 2 \mid m\\15 & 2 \nmid m \end{cases}.\]
We are then left to solve the equation $C(\mathbf{v})=0$ on this $13$-dimensional linear space which can be done by Theorem~\ref{thm.cubic_in_14}.

We have thus proved that $C(\mathbf{v}+t_1\mathbf{w}_1+\dots+t_m\mathbf{w}_m)=0$ identically in $t_1,\dots,t_m$ for some linearly independent vectors $\mathbf{v} \in K^s$ and $\mathbf{w}_1,\dots,\mathbf{w}_m \in \mathbb{Q}^s$.
By an observation of Lewis, this is enough to deduce the existence of a rational linear space of the same dimension, as we explain now, following an argument of Dietmann-Wooley \cite{dietmannwooley2003}.

Consider the $m$-dimensional $K$-rational spaces $V$ spanned by $\mathbf{v}$ and $\mathbf{w}_1,\dots,\mathbf{w}_m$ as well as $V^*$ spanned by $\mathbf{v}^*$ and $\mathbf{w}_1,\dots,\mathbf{w}_m$ where $^*$ denotes conjugation in $K$. If $\mathbf{v} \in \mathbb{Q}^s$ we are already done. Else, consider the $m+1$-dimensional space $W$ spanned by $\mathbf{v}$, $\mathbf{v}^*$ and $\mathbf{w}_1,\dots,\mathbf{w}_m$. If $C$ vanishes on $W$, we are also done as $W$ clearly contains a $m$-dimensional $\mathbb{Q}$-rational subspace. Else, by intersection theory 
the hypersurface defined by $C$ must intersect $W$ in a third $m$-dimensional $K$-rational subspace $L$. More precisely, by Theorem I.7.7 in Hartshorne~\cite{hartshorne2013algebraic} we have
\[
i(W,C;V)+i(W,C;V^*)+\sum_j i(W,C;Z_j)  \cdot \deg Z_j=(\deg W)(\deg C)=3
\]
where $i(W,C;V)$ denotes the intersection multiplicity and $Z_i$ are the other irreducible components of $C \cap W$. Since $W$ is invariant under conjugation, we must have $i(W,C;V)=i(W,C;V^*)$ and thus both numbers are equal to $1$, implying that there is a unique third component $L=Z_1$ which is then necessarily linear. Finally, since $W$ and $C$ are conjugation invariant, the three spaces $V$, $V^*$ and $L$ are permuted under conjugation and thus $L$ itself is conjugation invariant, i.e. the desired rational linear space. \qed

We remark that the use of intersection theory in the previous argument can be replaced by an explicit algebraic computation, as shown in Wooley~\cite{wooley1997}. \\

To deduce Theorem \ref{thm.primes}, we follow the strategy in \cite{birchgoldbach2010}. In particular, we show that the existence of a rational line implies the existence of almost prime solutions, regardless of the number of variables. We thus assume that for some linearly independent vectors $\mathbf{a},\mathbf{b} \in \mathbb{Z}^s$, we have $C(\mathbf{a}t+\mathbf{b}u)=0$ identically in $t$ and $u$. If $a_i=b_i=0$ for some $i$, then we can set $c_i=1$ and continue to work with the other variables. By taking a suitable linear combination, we can then assume that indeed all $a_i$ and $b_i$ are different from $0$. Rescaling $u$ by a factor of $a_1a_2\dots a_s$ and then rescaling the variables by a factor of $a_i$ (thereby changing $c_i$ by a factor of $a_i$), we may even assume that all the $a_i$ are equal to $1$, i.e.
\[C(t+b_1u,t+b_2u,\dots,t+b_nu)=0\]
identically in $t$ and $u$. By the Green--Tao Theorem \cite{greentao}, the primes contain infinitely many arithmetic progressions of length $2M+1$ where $M=2\max_i \vert b_i\vert+1$, i.e. there are infinitely many pairs $(\ell,d)$ such that $\ell+kd$ is prime for all $\vert k\vert \le M$. Choosing $t=\ell$ and $u=k$ then yields the desired result with $c_i=1$. \qed

\section{Algebraic Preliminaries} \label{sec.algebraic_prelim}
We now prepare for the proof of Theorem~\ref{thm.cubic_in_14}. While our main result is proved only for imaginary quadratic number fields we will introduce the matter in a general fashion and not restrict ourselves to these fields for now. We will aim to highlight whenever phenomena occur that set apart the situation for imaginary quadratic number fields from a general setting. In particular, even when $K/\Q$ is an imaginary quadratic number field we still sometimes prefer to write $n = [K \colon \Q]$. 

Let $K$ be a number field of degree $n$ over $\mathbb{Q}$ and denote by $\O$ its ring of integers. 

 Define the $\mathbb{R}$-vector space $K_{\mathbb{R}} \coloneqq K \otimes_{\mathbb{Q}} \mathbb{R}$ and note that we have natural embeddings $\O \subset K \subset K_{\mathbb{R}}$. The space $K_\R$ is sometimes referred to as the \emph{Minkowski space} of $K$. Note that there exist integers $n_1$ and $n_2$ with $n_1+n_2 = n$ such that $K$ admits $n_1$ real embeddings $\sigma_1, \hdots, \sigma_{n_1}$ and $2n_2$ complex embeddings $\sigma_{n_1+1}, \overline{\sigma}_{n_1+1}, \hdots, \sigma_{n_1+n_2}, \overline{\sigma}_{n_1+n_2}$ so that $K_{\mathbb{R}} \cong \mathbb{R}^{n_1} \times \mathbb{C}^{n_2}$.

Denote by $\pi_i$ the projection from $K_{\mathbb{R}} \cong \mathbb{R}^{n_1} \times \mathbb{C}^{n_2}$ to the $i$-th coordinate, which may take real or complex values.
 We define the trace map $\mathrm{tr} \colon K_{\mathbb{R}} \rightarrow \mathbb{R}$ and norm map $\mathrm{Norm} \colon K_{\mathbb{R}} \rightarrow \mathbb{R}$ as 
\[
\mathrm{tr}(\alpha) = \sum_{i=1}^{n_1} \pi_i(\alpha) + \sum_{i=n_1+1}^{n_2} \Re (\pi_i(\alpha)),
\]
and
\[
\mathrm{Norm}(\alpha) = \prod_{i=1}^{n_1} \left\lvert \pi_i(\alpha) \right\rvert \prod_{i=n_1+1}^{n_2} \left\lvert \pi_i(\alpha) \right\rvert^2,
\]
respectively. If $\alpha \in K$ then these are just the usual norm and trace function from algebraic number theory. 

Pick a basis $\Omega = \{\omega_1, \hdots, \omega_n\}$ of $\O$.
Any element $\alpha \in K_{\mathbb{R}}$ may be expressed in the form
$\alpha = \sum_{j=1}^n \alpha_j \omega_j$ for some $\alpha_j \in \mathbb{R}$. For such $\alpha$ we define a height \[\lvert \alpha \rvert \coloneqq \max_j \lvert \alpha_j  \rvert.\] Note that this depends on the choice of basis $\Omega$ for $\O$. Given a vector $\bm{\alpha} = (\alpha^{(1)}, \hdots, \alpha^{(s)}) \in K_\R^s$ we further denote
\[
|\bm{\alpha}| \coloneqq \max_k |\alpha^{(k)}|.
\]
We may alternatively define another height on $K_\R$ given by
\[
|\alpha|_K \coloneqq \max_p \left\lvert \pi_p(\alpha)\right\rvert.
\]
As noted by Pleasants~\cite[Section 2]{pleasants1975cubic} we have
\[
|\alpha| \asymp |\alpha|_K,
\]
for all $\alpha \in K_\R$.  If $\alpha, \beta \in K_\R$ then it is easy to see that this height satisfies
\begin{align*}
    |\alpha \beta|_K &\leq |\alpha|_K |\beta|_K, \\
    |\alpha + \beta|_K &\leq |\alpha|_K + |\beta|_K \\
    |\alpha^{-1} |_K &\leq \frac{|\alpha|_K^{n-1}}{\mathrm{Norm}(\alpha)}.
\end{align*}
The same inequalities therefore hold for $|\cdot|$ if we replace the symbols $\leq$ by $\ll_K$. 
It would be desirable to have the last inequality in the form $|\alpha^{-1}| \asymp |\alpha|^{-1}$ which would result if $\mathrm{Norm}(\alpha) \asymp |\alpha|^n$. However, if $\alpha$ is a unit in $\O$ then $\mathrm{Norm}(\alpha) = 1$ while the height $|\alpha|$ may be unbounded, at least whenever $K$ is not an imaginary quadratic number field. This is one of the points where our argument crucially depends on the latter assumption.

If $K = \Q(\sqrt{-d})$ is an imaginary quadratic number field then, depending on the value of the residue class of $d$ mod $4$ we can choose $\{1, \sqrt{-d}\}$ or $\{1, (1+\sqrt{-d})/2 \}$ as an integral basis for $\O$. We thus find that
\[
\mathrm{Norm}(\alpha) \asymp |\alpha|^2.
\]
In particular we find
\[
|\alpha^{-1}| \asymp |\alpha|^{-1}.
\]
Given an ideal $J \subset \O$ we recall that $\O/J$ is finite and we define as usual the norm of the ideal to be
\[
N(J) \coloneqq \# \left( \O/J \right).
\]
For a fractional ideal of $K$ this norm is, as usual, extended multiplicatively using the unique factorization into prime ideals inside $K$.
Given $\gamma \in K$ we further define the \emph{denominator ideal} of $\gamma$ as
\[
\mathfrak{a}_\gamma \coloneqq \left\{ x \in \O \colon x \gamma \in \mathcal{O} \right\}.
\]
As the name suggests, and this is not very difficult to verify, $\mathfrak{a}_\gamma$ is an ideal inside $\O$, contained in the fractional ideal $(\gamma)^{-1}$. We will need the following fact several times.
\begin{lemma} \label{lem.denominator_ideal_fixed}
Let $J \subset \O$ be an ideal. Then there are at most $N(J)$ different elements $\gamma \in K/\O$ such that $\mathfrak{a}_\gamma = J$.
\end{lemma}
\begin{proof}
To see this, note first that for any two fractional ideals $\mathfrak{b}, \mathfrak{c} \subset K$ with $\mathfrak{b} \supset \mathfrak{c}$ there exists some $d \in \O$ such that $d \mathfrak{b}, d \mathfrak{c} \subset \O$. Thus
\[
[\mathfrak{b} \colon \mathfrak{c}] = [d\mathfrak{b} \colon d\mathfrak{c}] = \frac{[\O \colon d \mathfrak{c}]}{ [\O \colon d \mathfrak{b}]} =  N(d\mathfrak{c})/N(d\mathfrak{b}) = N(\mathfrak{c})/N(\mathfrak{b}).
\]
Now note that if $\mathfrak{a}_\gamma = J$ we must have $\gamma \in J^{-1} \O$, where 
\[
J^{-1} = \{ x \in K \colon xJ \subset \O\}.
\]
But now  $[J^{-1}\O \colon \O] = N(J)$ and so the result follows.
\end{proof}

We shall further require a version of Dirichlet's approximation theorem. 
\begin{lemma}
Let $K/\Q$ be a number field of degree $n$. Let $\alpha \in K_\R$ and let $Q \geq 1$. Then there exist some $a,q \in \O$ with $1 \leq |q| \leq Q$ such that
\begin{equation} \label{eq.dir_approx_integral_form}
    \left\lvert q\alpha - a \right\rvert \leq \frac{1}{Q}.
\end{equation}
\end{lemma}
\begin{proof}
Consider the set  $\mathcal{Q}$ of algebraic integers given by 
\[
\mathcal{Q} = \left\{ \sum_{j} q_j \omega_j \in \O \colon 0 \leq q_j \leq Q\right\}.
\]
For any $q \in \mathcal{Q}$ we may express $q \alpha$ as 
\[
q\alpha = a_q + x_q,
\]
where $a_q \in \O$ and $x_q = \sum_j x_{q,j} \omega_j$ such that $0 \leq x_{q,j} < 1$ for $j = 1, \hdots, n$. We may partition $K_\R/\O = \left\{ \sum_j x_j \omega_j \colon 0 \leq x_j < 1\right\}$ into $Q^n$ boxes such that the height of the difference of two points in the same box is bounded by $1/Q$. Since $\mathcal{Q}$ has $(Q+1)^n$ elements, by the pigeonhole principle there must be $q_1,q_2 \in \mathcal{Q}$ such that $x_{q_1}$ and $x_{q_2}$ lie in the same box according to the partition above. Therefore we find
\[
|(q_1-q_2)\alpha - (a_{q_1}-a_{q_2})| = |x_{q_1} - x_{q_2}| \leq 1/Q.
\]
Taking $q = q_1-q_2$ and $a = a_{q_1}-a_{q_2}$ delivers the result.
\end{proof}


For the application to the mean-square averaging method introduced by Heath-Brown, we need a fractional form of Dirichlet's theorem. We are only able to obtain a satisfactory version for imaginary quadratic number fields, this being the first of the obstructions regarding possible generalizations mentioned in the introduction. Note that this is special to Heath-Brown's method and hence was not an issue in the work  of Ramanujam, Ryavec and Pleasants.

\begin{lemma} \label{lem.dir_approx_fractional_quadratic_imag}
Let $K/\Q$ be an imaginary quadratic number field (in particular $n=2$). Let $\alpha \in K_\R$ and let $Q \geq 1$. Then there exists some $\gamma \in K$ with $N(\mathfrak{a}_\gamma) \leq Q^n$ such that
\begin{equation} \label{eq.dirichlet_approx_fractional_quadr_imag}
    \left\lvert \alpha - \gamma \right\rvert \ll \frac{1}{N(\mathfrak{a}_\gamma)^{\frac{1}{n}}Q }.
\end{equation}
\end{lemma}
\begin{proof}
From Lemma~\ref{lem.dir_approx_fractional_quadratic_imag} we find that there exist $a, q\in \O$ with $|q| \leq Q$ such that
\[
|q\alpha - a| \leq 1/Q.
\]
Set $\gamma = a/q \in K$ and note that $(q) \subseteq \mathfrak{a}_\gamma$. In particular from this it follows that
\[
N(\mathfrak{a}_\gamma) \leq N((q)) = \mathrm{Norm}(q) \asymp |q|^n,
\]
where the last estimate is true since $K$ is an imaginary quadratic number field.
Thus
\[
|q|^{-1} \ll N(\mathfrak{a}_\gamma)^{-1/n},
\]
and so we  obtain 
\[
\left\lvert \alpha - \gamma \right\rvert \ll |q|^{-1} |q\alpha - a| \ll \frac{1}{N(\mathfrak{a}_\gamma)^{\frac{1}{n}}Q },
\]
as desired.
\end{proof}

We shall sometimes require the following easy lemma.
\begin{lemma} \label{lem.elements_in_ideals}
Let $J \subset \O$ be an ideal. Then there exist constants $c_1, c_2$ only depending on $K$ such that for any non-zero $g \in J$ we have
\[
c_1 N(J)^{1/n} \leq |g|,
\]
and we may always find a non-zero element $a \in J$ such that
\[
|a| \leq c_2 N(J)^{1/n}.
\]
\end{lemma}
\begin{proof}
First note that if $g \in J$ then $(g) \subset J$ and therefore 
\[
N(J) \leq N((g)) = \mathrm{Norm}(g) \ll |g|^n.
\]
For the second inequality note that there are at least $N(J)+1$ algebraic integers whose height does not exceed $N(J)^{1/n}$. By definition $N(J) = \#(\O/J)$ and hence at least two of these integers must lie in the same residue class modulo $J$. Their difference is therefore an algebraic integer $a \in J$ with $|a| \leq 2 N(J)^{1/n}$. 
\end{proof}

Finally we will also need the following.
\begin{lemma} \label{lem.trace_integral_discriminant}
Let $K/\Q$ be a number field and let $\Delta$ be the discriminant of this extension. Let $\alpha \in K_\R$ and assume that $\{\omega_i\}_i$ is an integral basis for $\O$. If
\[
\Delta^{-1}\mathrm{tr}(\alpha \omega_i) \in \mathbb{Z}
\]
holds for all $i = 1, \hdots, n$ then $\alpha \in \O$.
\end{lemma}
\begin{proof}
Write $\alpha = \sum_{j=1}^n \alpha_j \omega_j$, where $\alpha_j \in \R$. Due to the additivity of the trace we have
\[
\mathrm{tr}(\alpha \omega_i) = \sum_{j=1}^n \alpha_j \mathrm{tr}(\omega_i \omega_j).
\]
Denote by $\mathbf{T}$ the trace form, that is, the $n \times n$ matrix with entries $\mathrm{tr}(\omega_i \omega_j)$. Then if we identify $\mathbf{\alpha} = (\alpha_1, \hdots, \alpha_n) \in \Z^n$, the assumption of the lemma is equivalent to
\[
\Delta^{-1} \mathbf{T}(\mathbf{\alpha})  \in \Z^n.
\]
By definition $\det \mathbf{T} = \Delta$. Hence $\mathbf{T}':=\Delta \mathbf{T}^{-1}$ has integer entries. 
Combining this with our previous observation yields
\[
\mathbf{\alpha} = \mathbf{T}^{-1}\mathbf{T}(\mathbf{\alpha}) = \mathbf{T}'(\Delta^{-1} \mathbf{T}(\mathbf{\alpha})) \in \Z^n.
\]
Hence $\alpha \in \O$ as required.
\end{proof}

\section{The Dichotomy} \label{sec.dichotomy}
Let $C \in \O[x_1, \hdots, x_s]$ be a homogeneous cubic form. Our goal is to show that there always exists a non-trivial solution to $C = 0$ over $K$ provided $s \geq 14$ and $K$ is an imaginary quadratic number field. We follow the strategy of Davenport that was later refined by Heath-Brown~\cite{heath2007cubic}: Either $C$ represents zero non-trivially for 'geometric reasons', or we can establish an asymptotic formula for the number of solutions of bounded height, using the circle method.

\subsection{Davenport's Geometric Condition}

We may express $C(\mathbf{x})$ as
\[
C(\mathbf{x}) = \sum_{i,j,k} c_{ijk} x_i x_j x_k,
\]
where the coefficients $c_{ijk}$ are fully symmetric in the indices and lie in $\O$, after replacing $C(\mathbf{x})$ by $6C(\mathbf{x})$ if required. For $i = 1, \hdots, s$ further define the bilinear forms $B_i(\mathbf{x},\mathbf{y})$ by
\[
B_i(\mathbf{x},\mathbf{y}) = \sum_{j,k} c_{ijk} x_j y_k.
\]
Finally, we also consider an $s \times s$ matrix $M(\mathbf{x})$, the \emph{Hessian }of $C(\mathbf{x})$, whose entries are defined by
\[
M(\mathbf{x})_{jk} = \sum_i c_{ijk} x_i,
\]
so that 
\[
(M(\mathbf{x})\mathbf{y})_i = B_i(\mathbf{x},\mathbf{y}).
\]
We note that the entries are linear forms in the variables $\mathbf{x}$. Denote the rank of the matrix by
\[
r(\mathbf{x}) = \mathrm{rank}(M(\mathbf{x})). 
\]
As in Davenport's and Heath-Brown's work we obtain a dichotomy. 
\begin{lemma} \label{lem.geometric_condition}
 One of the following two alternatives holds.
\begin{enumerate}
    \item Davenport's Geometric Condition: For every integer $0 \leq r \leq s$ we have
    \begin{equation} \label{eq.rank_condition}
        \# \{ \mathbf{x} \in \O^s \colon \lvert{\mathbf{x}}\rvert < H, \; r(\mathbf{x}) = r \} \ll H^{nr}.
    \end{equation}
    \item The cubic form $C(\mathbf{x})$ has a non-trivial zero in $\O$.
\end{enumerate}
\end{lemma}
\begin{proof}
Consider the least integer $h=h(C)$ such that the cubic form may be written as 
\[
C(\mathbf{x}) = \sum_{i=1}^h L_i(\mathbf{x}) Q_i(\mathbf{x}),
\]
where $L_i$ are linear and $Q_i$ are quadratic forms defined over $K$. This is the $h$-invariant of $C$. It is easy to see that $1 \leq h \leq s$ holds, and that $C(\mathbf{x}) = 0$ has a non-trivial solution over $K$ if and only if $h < s$. 

We will show that if $h = s$ then alternative (1) holds. In fact, Pleasants~\cite[Lemma 3.5]{pleasants1975cubic} showed that the number of points $\mathbf{x} \in \O^s$ such that $|\mathbf{x}| < H$ holds, for which the equations $B_i(\mathbf{x}, \mathbf{y}) = 0$, $j = 1, \hdots, s$ have exactly $s-r$ linearly independent solutions $\mathbf{y}$ is bounded by $O(H^{n(s-h+r)})$. Hence taking $h = s$ delivers the desired bound~\eqref{eq.rank_condition}.
\end{proof}

We will henceforth assume that Davenport's Geometric Condition~\eqref{eq.rank_condition} is satisfied and apply the circle method. In particular as in~\cite{heath2007cubic} this condition implies that we have
    \begin{equation} \label{eq.rank_condition_for_bilinear}
    \# \{ \mathbf{x}, \mathbf{y} \in \O^s \colon |\mathbf{x}|, |\mathbf{y}| < H, B_i(\mathbf{x},\mathbf{y}) = 0, \forall i  \} \ll H^{ns},
    \end{equation}
for any $H \geq 1$.

\subsection{The Circle Method}

Let $\mathcal{B} \subset K_{\mathbb{R}}^s \cong \mathbb{R}^{ns}$ be a box of the form
\[
\mathcal{B} = \left\{ \left(\sum_j \alpha_{ij} \omega_j\right)_i \in K_{\mathbb{R}}^s \colon b_{ij}^{-} \leq \alpha_{ij} \leq b_{ij}^+ \right\},
\]
where $b_{ij}^- < b_{ij}^+$ are some real numbers 
.
For $P \geq 1$ consider the counting function
\[
N(P;\mathcal{B}) = N(P) = \left\{ \mathbf{x} \in P\mathcal{B} \cap \O^s \colon C(\mathbf{x}) = 0 \right\}.
\]
For $\alpha \in K_{\mathbb{R}}$ and $P \geq 1$ we define the exponential sum
\[
S(\alpha) = S(\alpha;P)  = \sum_{\mathbf{x} \in P\mathcal{B} \cap \O^s} e\left( \mathrm{tr}(\alpha C(\mathbf{x})) \right).
\]
Denote by $I \subset K_{\mathbb{R}}$ the set given by
\[
I = \left\{ \sum_{j=1}^n \alpha_j \omega_j \colon 0 \leq \alpha_j \leq 1 \right\},
\]
which may also be regarded as $K_{\R}/\O$. 
Due to orthogonality of characters we obtain
\[
N(P) = \int_{\alpha \in I} S(\alpha) d \alpha.
\]
We are now able to state the main technical theorem of our paper.
\begin{theorem} \label{thm.asymptotic}
Let $K/\Q$ be an imaginary quadratic number field and let $C(\mathbf{x})$ be a cubic form in $s \geq 14$ variables over $K$. Suppose that $C(\mathbf{x})$ is irreducible over $K$ and that Davenport's Geometric Condition~\eqref{eq.rank_condition} is satisfied. Then we have the asymptotic formula
\[
N(P) = \sigma P^{n(s-3)} + o\left(P^{n(s-3)}\right), \quad \text{as} \quad P \rightarrow \infty,
\]
where $\sigma>0$ is the product of the usual singular integral and singular series.
\end{theorem}
Therefore Theorem~\ref{thm.cubic_in_14} follows directly from Lemma~\ref{lem.geometric_condition} and Theorem~\ref{thm.asymptotic} where we also note that  a reducible cubic form always
contains a linear factor over $K$ and therefore has a non-trivial solution for obvious reasons.
\subsection{The major arcs} \label{sec.major_arcs_cubic}
For this section we do not need to assume that $K$ is an imaginary quadratic number field of $\Q$.
As in Pleasants, we choose as center of our box $\mathcal{B}=\mathcal{B}(\mathbf{z})$ a solution $\mathbf{z} \in K_{\mathbb{R}}$ of $C(\mathbf{z})=0$ satisfying $\frac{\partial C}{\partial x_1}(\mathbf{z}) \ne 0$ and $z_1, \hdots, z_n \neq 0$. Such a vector $\mathbf{z}$ always exists by~\cite[Lemma 7.2]{pleasants1975cubic} provided $C$ is irreducible.

Let $\gamma \in K/\mathcal{O}$ and define
\[
\mathfrak{M}_\gamma \coloneqq \left\{ \alpha \in I \colon |\alpha-\gamma| < P^{-3+\nu} \right\},
\]
where we regard $I = K_\R/\mathcal{O}$. We define the \emph{major arcs} as
\[
\mathfrak{M} = \bigcup_{\substack{\gamma \in K/\mathcal{O} \\ N(\mathfrak{a}_\gamma) \leq P^\nu}} \mathfrak{M}_\gamma,
\]
and the \emph{minor arcs} as
\[
\mathfrak{m} = I \setminus \mathfrak{M}.
\]
Further, define the sum $S_\gamma$ via
\[
S_\gamma = \sum_{\mathbf{x} \; \mathrm{mod} \; N(\mathfrak{a}_\gamma)} e(\tr({\gamma C(\mathbf{x})}).
\]
Given a parameter $R \geq 1$ we define the \emph{truncated singular series} to be
\[
\mathfrak{S}(R) \coloneqq \sum_{\substack{\gamma \in K/\O \\ N(\mathfrak{a}_\gamma) \leq R}} N(\mathfrak{a}_\gamma)^{-ns}S_\gamma,
\]
and the \emph{truncated singular integral} to be 
\[
\mathfrak{I}(R) \coloneqq \int_{|\zeta| < R^\nu} \int_{\mathcal{B}} e(\mathrm{tr}(\zeta R^{-3} C(R \bm{\xi})) ) d\bm{\xi} d\zeta.
\]
Pleasants~\cite[Lemma 7.1]{pleasants1975cubic} shows that if $\nu < \frac{1}{n+4}$ is satisfied then we have
\[
\int_{\mathfrak{M}} S(\alpha) d\alpha = \mathfrak{S}(P^\nu) \mathfrak{I}(P) P^{n(s-3)} + o(P^{n(s-3)}).
\]
Moreover, if $\mathcal{B} = \mathcal{B}(\mathbf{z})$ is the box as in the beginning of the section, provided that the sidelengths of the boxes are sufficiently small, and if $C(\mathbf{x})$ is irreducible over $K$ then Pleasants~\cite[Lemma 7.2]{pleasants1975cubic} further shows that $\mathfrak{I}(R)$ converges absolutely to a positive number $\mathfrak{I}$.

We remark that Lemma 7.2 in~\cite{pleasants1975cubic} was originally stated under the weaker assumption that $C(\mathbf{x})$ is not a rational multiple of a cube of a linear form. His proof relies on a result by Davenport~\cite[Lemma 6.2]{davenport_32_variables}, which assumes the existence of a non-singular, real solution $\bm{\xi} \in \R^n$ of a rational cubic form $G$ such that 
\[
\frac{\partial G}{\partial x_i}(\bm{\xi}) \neq 0, \quad \xi_i \neq 0,
\]
holds for some $i$. In particular Pleasants writes that \textit{"this hypothesis is not used in the proof of the lemma, however, and in any case the argument that follows could easily be adapted to provide it"}. While one can always find  $\bm{\xi} \in \R^s$ with $\frac{\partial G}{\partial x_i}(\bm{\xi}) \neq 0$ unless $G$ is a rational multiple of a cube of a linear form, one can not necessarily ensure that $\xi_i \neq 0$ for the same index $i$. Consider for example $G(x_1,\hdots, x_n) = x_1(x_2^2+\cdots +x_n^2)$. It is possible that Davenport's result~\cite[Lemma 6.2]{davenport_32_variables} holds nevertheless in this generality but at least the standard method of establishing bounded variation of the auxiliary function involved in the proof by showing the existence of right and left derivatives, see for example~\cite[Lemma 16.1]{davenport_book}, fails in general.

The singular series $\mathfrak{S}(R)$ may or may not converge absolutely as $R \rightarrow \infty$. If it does converge, then provided non-singular $\mathfrak{p}$-adic solutions of $C(\mathbf{x}) =0$ exist for all primes $\mathfrak{p}$, by standard arguments it follows that $\mathfrak{S} >0$. See for example the proof of Lemma 7.4 in~\cite{pleasants1975cubic}, where this argumentation is carried out in our setting. Finally, Lewis~\cite{lewis_p_adic_zeroes} showed that these non-singular $\mathfrak{p}$-adic solutions always exist whenever $s \geq 10$. Therefore we obtain the following.
\begin{theorem} \label{thm.major_arcs_if_sing_converges}
Let $C \in \O[x_1, \hdots, x_s]$ be an irreducible cubic form. Assume that $s \geq 10$. If the singular series $\mathfrak{S}(R)$ converges absolutely as $R \rightarrow \infty$ then 
\[
\int_{\mathfrak{M}} S(\alpha) d\alpha = \sigma P^{n(s-3)} + o(P^{n(s-3)}),
\]
for some $\sigma > 0$ as $P \rightarrow \infty$.
\end{theorem}
In particular, in Section~\ref{sec.vdc_pointwise_sing_series} we will establish the following. 
\begin{theorem} \label{thm.sing_series_in_13}
Assume that $s \geq 13$ and that Davenport's Geometric Condition~\eqref{eq.rank_condition} is satisfied then the singular series converges absolutely. Therefore if $C(\bm{x})$ is irreducible we have 
\[
\int_{\mathfrak{M}} S(\alpha) d\alpha = \sigma P^{n(s-3)} + o\left(P^{n(s-3)}\right),
\]
for some $\sigma > 0$ as $P \rightarrow \infty$.
\end{theorem}
We remark that we show this result for any number field $K$.

\section{Auxiliary Diophantine Inequalities} \label{sec.aux_dioph_ineq}
To bound the Weyl sum $S(\alpha)$ of a general cubic form, classical Weyl differencing leaves us with the task of examining the number of solutions to certain auxiliary Diophantine inequalities. Davenport's crucial idea was to bootstrap these inequalities using his \textit{Shrinking Lemma}, combined with the observation that sufficiently strong Diophantine inequalities already imply divisibility or even equality.

In this section, we prepare these arguments by providing a version of this observation adapted to our setting. We are only able to show a satisfactory version of this lemma if $K/\Q$ is an imaginary quadratic number field, this being the second of the obstructions mentioned in the introduction.

\begin{lemma} \label{lem.small_torus_norm_implies_divisibility}
Assume that $K/\Q$ is a number field and denote by $\Delta$ the discriminant of this extension. There exists a real positive constant $A>0$ depending only on $K$ and the choice of integral basis $\Omega$ for $K$ such that the following statement holds. 

Let $M \geq0$ be a real number and let $\alpha \in K_\R$. Suppose that $\alpha = \gamma + \theta$ with $\gamma \in K$ and $M|\theta| N(\mathfrak{a}_\gamma)^{1/n} \leq A$. If $m \in \O$ is such that $|m| \leq M$ and $\norm{ \Delta^{-1}\mathrm{tr}({\alpha m \omega_j}) } < P_0^{-1}$ holds for all $j = 1, \hdots, n$ where $A P_0 \geq N(\mathfrak{a}_\gamma)^{1/n}$ then $m \in \mathfrak{a}_\gamma$. In particular if either of the conditions
\begin{enumerate}
    \item $M \leq A N(\mathfrak{a}_\gamma)^{1/n}$, or
    \item  $K$ is an imaginary quadratic number field and $A |\theta| \geq  N(\mathfrak{a}_\gamma)^{-1/n} P_0^{-1}$ 
\end{enumerate}
is satisfied, then we must have $m = 0$. 
\end{lemma}
\begin{proof}
Note first that
\[
\norm{\Delta^{-1}\mathrm{tr}({\gamma m \omega_j})} \leq \norm{\Delta^{-1}\mathrm{tr}({\alpha m \omega_j})}  + \norm{\Delta^{-1}\mathrm{tr}({\theta m \omega_j})}.
\]
Now due to our assumption we have $\norm{\Delta^{-1}\mathrm{tr}({\alpha m \omega_j})} < P_0^{-1}$. Further it is easy to see that
\[
 \Delta^{-1}|\mathrm{tr}(\theta m \omega_j)| \ll |\theta| M.
\]
Therefore choosing $A$ sufficiently small we find
\begin{equation} \label{eq.trace_small}
    \norm{\Delta^{-1}\mathrm{tr}({\gamma m \omega_j})} < \frac{A^{1/2}}{ N(\mathfrak{a}_\gamma)^{1/n} },
\end{equation}
for all $j = 1, \hdots, n$. As before write  $\mathbf{T} = (\mathrm{tr}(\omega_i \omega_j))_{i,j}$ for the trace form.
Write $\mathbf{x} \in \R^n$ for the real vector obtained from $\gamma m$ under the isomorphism $K_\R \cong \R^n$. Then~\eqref{eq.trace_small} is equivalent to saying that there exist $\mathbf{a} \in \Z^n$ and $\mathbf{r} \in \R^n$ with $|\mathbf{r}| < \frac{A^{1/2}}{ N(\mathfrak{a}_\gamma)^{1/n} }$ such that
\[
 \mathbf{T}(\Delta^{-1}\mathbf{x}) = \mathbf{a} + \mathbf{r}.
\]
Recall that $\Delta \mathbf{T}^{-1}$ is an integral matrix whose entries are bounded in terms of $K$. Therefore 
\[
 \mathbf{x} = \Delta \mathbf{T}^{-1}(\mathbf{a})  + \Delta \mathbf{T}^{-1}(\mathbf{r}).
\]
Now $ \mathbf{T}^{-1}(\mathbf{a}) \in \Z^n$ and 
\[
|\Delta \mathbf{T}^{-1}(\mathbf{r})| < \frac{A^{1/3}}{N(\mathfrak{a}_\gamma)^{1/n}}, 
\]
after decreasing $A$ if necessary.
We thus find that 
\[
 \gamma m = a + \rho,
\]
where $a \in \O$ and $|\rho| < \frac{A^{1/3}}{N(\mathfrak{a}_\gamma)^{1/n}}$.
By Lemma~\ref{lem.elements_in_ideals} there exists $g \in \mathfrak{a}_\gamma$ with $|g| \asymp N(\mathfrak{a}_\gamma)^{1/n}$. From the above equation we see that $g \rho \in \O$, and so, unless $\rho = 0$ we have
\[
1 \leq |g \rho|  < A^{1/4},
\]
after decreasing $A$ if necessary.
Choosing $A$ suitably small therefore leads to a contradiction whence we must have $\rho = 0$, and so $m \in \mathfrak{a}_\gamma$. This finishes the first part of the proof.

If we now assume that $M \leq A N(\mathfrak{a}_\gamma)^{1/n}$ is satisfied then by choosing $A$ suitably small this implies that $m = 0$ via Lemma~\ref{lem.elements_in_ideals}. 

Finally, assume that $A |\theta| >  (N(\mathfrak{a}_\gamma)^{1/n} P_0)^{-1}$ is satisfied and that $K$ is an imaginary quadratic number field. Upon  choosing $A$ even smaller if necessary, we find that
\[
\Delta^{-1} |\mathrm{tr}(\theta m \omega_j)| \leq \frac{1}{2},
\]
for all $j = 1, \hdots, n$ and thus
\[
\Delta^{-1} |\mathrm{tr}(\theta m \omega_j)| = \norm{\Delta^{-1} \mathrm{tr}(\theta m \omega_j)} \leq\norm{\Delta^{-1} \mathrm{tr}(\gamma m \omega_j)} + \norm{\Delta^{-1} \mathrm{tr}(\alpha m \omega_j)} < P_0^{-1},
\]
for all $j = 1, \hdots, n$. Write  $\mathbf{y} = (y_1, \hdots, y_n)$ for the image of $\theta m $ under the isomorphism $K_\R \cong \R^n$ and let $\mathbf{T}$ be the trace form as above.
The above inequality is equivalent to saying that there exists some $\mathbf{t} \in \R^n$ with $|\mathbf{t}|<  P_0^{-1}$ such that
\[
\mathbf{T}(\Delta^{-1}\mathbf{y}) = \mathbf{t}.
\]
As before the inverse of $\mathbf{T}$ is a matrix with rational entries, whose absolute value is bounded by $O(1)$. Hence
\[
|\mathbf{y}| = \Delta |\mathbf{T}^{-1} (\mathbf{t}) | \ll  |\mathbf{t}| <  P_0^{-1}.
\]
Further $|\mathbf{y}| = |\theta m|$, and since $K$ is an imaginary quadratic number field we have $|\theta^{-1}| \asymp |\theta|^{-1}$ and so
\[
|m| \ll  (P_0 |\theta|)^{-1}.
\]
Hence for sufficiently small $A$ we obtain
\[
|m| <A^{1/2} N(\mathfrak{a}_\gamma)^{1/n}.
\]
Choosing $A$ to be suitably small implies $m = 0$ by Lemma~\ref{lem.elements_in_ideals}.
\end{proof}

We now recall Davenport's shrinking lemma~\cite[Lemma 12.6]{davenport_book}.
\begin{lemma} \label{lem.shrinking}
Let $L \colon \R^m \rightarrow \R^m$ be a linear map. Let $a >0$ be a real number and for a real number $Z >0$ consider
\[
N(Z) = \left\{ \mathbf{u} \in \Z^m \colon |\mathbf{u}| < aZ, \; \norm{(L(\mathbf{u}))_i} < a^{-1} Z, \text{for all $i$}  \right\}.
\]
Then if $0 < Z \leq 1$ we have
\[
N(1) \ll_m Z^{-m} N(Z).
\]
\end{lemma}

As noted in~\cite{heath2007cubic} the lemma was originally only stated when $a \geq 1$ but we may extend the range of $a$ to all positive real numbers since the result holds trivially if $0 < a < 1$.

\section{Weyl Differencing} 

\label{sec.weyl_differencing_cubes}
One of the main innovations in~\cite{heath2007cubic} is to introduce an averaged van der Corput differencing approach in order to bound the contribution from the minor arcs. Since this cannot handle the entire range of minor arcs we need to supplement it with an estimate coming from conventional Weyl differencing.

Let $\alpha \in K_\R$. Throughout this section we will write
\[
\alpha = \gamma + \theta,
\]
where $\gamma \in K$ and $\theta \in K_\R$.
Note as in~\cite[Lemma 2.1]{pleasants1975cubic} we find
\begin{equation} \label{eq.sum_after_weyl_diff}
    |S(\alpha)|^4 \ll P^{ns} \sum_{|\mathbf{x}|, |\mathbf{y}| < P} \prod_{i=1}^s \prod_{j=1}^n \min\left( P, \norm{\mathrm{tr}(6 \alpha \omega_j B_i(\mathbf{x},\mathbf{y}))}^{-1} \right).
\end{equation}
This estimate is proved using a classical Weyl differencing procedure adjusted to this context. Following standard arguments as in Davenport~\cite[Chapter 13]{davenport_book} we now transform this into a counting problem.

Given $\alpha \in \R$ and $P \geq 1$ define
\[
N(\alpha, P) \coloneqq \# \left\{(\mathbf{x}, \mathbf{y}) \in \O^{2s} \colon |\mathbf{x}| < P, \, |\mathbf{y}| < P, \, \norm{\mathrm{tr}(6 \alpha \omega_j B_i(\mathbf{x},\mathbf{y}))} < P^{-1}, \, \forall i,j \right\}.
\]
For a fixed $\mathbf{x} \in \O^s$ write further
\[
N(\mathbf{x}) \coloneqq \# \left\{ \mathbf{y} \in \O^{s} \colon  |\mathbf{y}| < P, \, \norm{\mathrm{tr}(6 \alpha \omega_j B_i(\mathbf{x},\mathbf{y}))} < P^{-1}, \, \forall i,j \right\},
\]
so that
\[
N(\alpha,P) = \sum_{|\mathbf{x}| < P} N(\mathbf{x}).
\]
Let $r_{ij}$ be integers such that $0 \leq r_{ij} <P$ for $i = 1, \hdots, s$, $j = 1, \hdots, n$. We claim that there exist no more than $N(\mathbf{x})$ integer tuples $\mathbf{y} \in \O^s$, which lie in a box whose edges have sidelengths at most $P$ such that
\[
\frac{r_{ij}}{P} \leq \left\{\mathrm{tr}(6 \alpha \omega_j B_i(\mathbf{x},\mathbf{y})) \right\} < \frac{r_{ij}+1}{P}
\]
is satisfied for all $i = 1, \hdots, s$ and $j = 1, \hdots, n$, where $\{x\}$ denotes the fractional part of a real number $x$. Indeed, if $\mathbf{y}_1$ and $\mathbf{y}_2$ are two such integer tuples satisfying the above system of inequalities then $|\mathbf{y}_1 - \mathbf{y}_2| < P$ and 
\[
\norm{\mathrm{tr}(6 \alpha \omega_j B_i(\mathbf{x},\mathbf{y}_1-\mathbf{y}_2))} < P^{-1}
\]
holds for all $i,j$. Hence, since $\mathbf{y} = \mathbf{0}$ is a possible solution, there are no more than $N(\mathbf{x})$ possible solutions to the system of inequalities above. Dividing the box $P \mathcal{B}$ into $2^{ns}$ boxes whose edges have sidelength  at most $P$ we find
\begin{align*}
    \sum_{|\mathbf{y}| < P} \prod_{i=1}^s \prod_{j=1}^i \min\left( P, \norm{\mathrm{tr}(6 \alpha \omega_j B_i(\mathbf{x},\mathbf{y}))}^{-1} \right) &\ll N(\mathbf{x}) \prod_{i,j} \sum_{r_{ij} = 0}^{P} \min\left( P, \frac{P}{r_{ij}}, \frac{P}{P-r_{ij}-1} \right) \\
    &\ll N(\mathbf{x}) (P \log P)^{ns}.
\end{align*}
Upon summing this estimate over $|\mathbf{x}| <P$ and using~\eqref{eq.sum_after_weyl_diff}  we obtain
\begin{equation} \label{eq.exp_sum_bound_solns_bilinears}
    |S(\alpha)|^4 \ll P^{2ns} (\log P)^{ns} N(\alpha,P).
\end{equation}
We now proceed to estimate $N(\alpha,P)$ using the results from the previous section.

For fixed $\mathbf{x} \in \O^s$ identifying $\O^s \cong \R^{ns}$ and given $\mathbf{y} \in \O^s$ one may view the map
\[
\mathbf{y} \mapsto (\mathrm{tr}(6 \alpha \omega_j B_i(\mathbf{x},\mathbf{y})))_{i,j}
\]
as a linear map $\R^{ns} \rightarrow \R^{ns}$. Hence we can apply Lemma~\ref{lem.shrinking} where $N(\mathbf{x}) = N(1)$ in the notation of the lemma where $Z$ is to be determined in due course. Summing over the $|\mathbf{x}| < P$ then yields
\begin{multline} \label{eq.first_shrinkage}
N(\alpha, P) \ll 
Z^{-ns} \# \left\{(\mathbf{x}, \mathbf{y}) \in \O^{2s} \colon |\mathbf{x}| < P, \, |\mathbf{y}| < ZP, \, \right. \\
\left. \norm{\mathrm{tr}(6 \alpha \omega_j B_i(\mathbf{x},\mathbf{y}))} < ZP^{-1}, \, \forall i,j \right\}.
\end{multline}
If we apply the same procedure to the quantity on the right hand side of~\eqref{eq.first_shrinkage}, but now with the roles of $\mathbf{x}$ and $\mathbf{y}$ reversed we obtain
\begin{multline} \label{eq.second_shrinkage}
N(\alpha, P) \ll 
Z^{-2ns} \# \left\{(\mathbf{x}, \mathbf{y}) \in \O^{2s} \colon |\mathbf{x}| < ZP, \, |\mathbf{y}| < ZP, \, \right. \\
\left. \norm{\mathrm{tr}(6 \alpha \omega_j B_i(\mathbf{x},\mathbf{y}))} < Z^2P^{-1}, \, \forall i,j \right\}.
\end{multline}
At this point we will employ Lemma~\ref{lem.small_torus_norm_implies_divisibility}. We wish to choose $Z$ such that the bilinear forms appearing in the right hand side of~\eqref{eq.second_shrinkage} are forced to vanish. To this end, in the notation of the lemma we take $m = 6 \Delta B_i(\mathbf{x},\mathbf{y})$, $M \asymp 6Z^2P^2$ and $P_0^{-1} = Z^2P^{-1}$. Choose the parameter $Z$ so that it satisfies
\[
0<Z<1, \quad Z^2 \ll  (P^2 |\theta| N(\mathfrak{a}_\gamma)^{1/n})^{-1}, \quad Z^2 \ll  \frac{P}{N(\mathfrak{a}_\gamma)^{1/n}},
\]
as well as
\[
Z^2 \ll  \max \left( \frac{N(\mathfrak{a}_\gamma)^{1/n}}{P^2}, N(\mathfrak{a}_\gamma)^{1/n} |\theta| P \right),
\]
where the implicit constants involved are sufficiently small such that the assumptions of Lemma~\ref{lem.small_torus_norm_implies_divisibility} are satisfied.
Provided $K$ is an imaginary quadratic number field, Lemma~\ref{lem.small_torus_norm_implies_divisibility} and~\eqref{eq.second_shrinkage} give
\begin{equation*}
    N(\alpha, P) \ll Z^{-2ns} \left\{(\mathbf{x}, \mathbf{y}) \in \O^{2s} \colon |\mathbf{x}| < ZP, \, |\mathbf{y}| < ZP, \, B_i(\mathbf{x},\mathbf{y}) = 0, \; i = 1, \hdots, s \right\},
\end{equation*}
where we note that clearly $6 \Delta B_i(\mathbf{x},\mathbf{y}) = 0$ if and only if $ B_i(\mathbf{x},\mathbf{y}) = 0$.

Since we assume that Davenport's Geometric Condition~\eqref{eq.rank_condition} is satisfied it follows from the simple observation~\eqref{eq.rank_condition_for_bilinear} that 
\begin{equation*}
    N(\alpha,P) \ll Z^{-2ns} (ZP)^{ns}.
\end{equation*}
From~\eqref{eq.exp_sum_bound_solns_bilinears} for permissible $Z$ as described above we therefore have
\begin{equation} \label{eq.exp_sum_in_terms_of_parameter_Z}
    |S(\alpha)|^4 \ll P^{3ns+\varepsilon} Z^{-ns}.
\end{equation}
The estimate is optimised when $Z$ is as large as possible. Hence if we take
\begin{equation*}
    Z^2 \asymp  \min\left\{1, (P^2 |\theta| N(\mathfrak{a}_\gamma)^{1/n})^{-1}, \frac{P}{N(\mathfrak{a}_\gamma)^{1/n}}, 
   \max \left( \frac{N(\mathfrak{a}_\gamma)^{1/n}}{P^2}, N(\mathfrak{a}_\gamma)^{1/n} |\theta| P \right) \right\}
\end{equation*}
then $Z$ is clearly in the permissible range, and we deduce 
\begin{multline*}
    |S(\alpha)|^4 \ll P^{3ns+\varepsilon} \left( 1 +P^2 |\theta| N(\mathfrak{a}_\gamma)^{1/n} +  {P}^{-1}{N(\mathfrak{a}_\gamma)^{1/n}}  \right. \\
    \left. + \min \left(P N(\mathfrak{a}_\gamma)^{-1/n} , \left(N(\mathfrak{a}_\gamma)^{1/n} |\theta| P\right)^{-1} \right) \right)^{\frac{ns}{2}}.
\end{multline*}
In particular, if $N(\mathfrak{a}_\gamma)^{1/n} \leq P^{3/2}$ then $P^{-1} N (\mathfrak{a}_\gamma)^{1/n} \leq P^{1/2}$ and so we find
\[
|S(\alpha)| \ll P^{ns+\varepsilon} \left( N(\mathfrak{a}_\gamma)^{1/n} |\theta|  + (N(\mathfrak{a}_\gamma)^{1/n} |\theta| P^3)^{-1} + P^{-3/2} \right)^{\frac{ns}{8}}
\]
in this case.
Finally since $X^{1/2} \leq X/Y + Y$ for any two positive real numbers $X$ and $Y$ we see that the last term of the right hand side above is dominated by the other two summands. We summarise the main result of this section.
\begin{lemma} \label{lem.weyl_estimate}
Let $K/\Q$ be an imaginary quadratic number field. Let $\alpha \in K_\R$ and write $\alpha = \gamma + \theta$ where $\gamma \in K$ and $\theta \in K_\R$. If $N(\mathfrak{a}_\gamma)^{1/n} \leq P^{3/2}$ then we have
\begin{equation} \label{eq.weyl_bound}
    S(\alpha) \ll P^{ns+\varepsilon} \left( N(\mathfrak{a}_\gamma)^{1/n} |\theta|  + (N(\mathfrak{a}_\gamma)^{1/n} |\theta| P^3)^{-1} \right)^{\frac{ns}{8}}. 
\end{equation}
\end{lemma}
This bound will be useful for the range in the minor arcs when the parameter $\theta$ is small.

\section{Pointwise van der Corput Differencing and the singular series} \label{sec.vdc_pointwise_sing_series}
In this section we will perform a pointwise van der Corput differencing argument, in order to show that the singular series converges absolutely. This argument works over a general number field. We start by considering the exponential sum $S(\gamma)$, where $\gamma \in K$ and we set $P = N(\mathfrak{a}_\gamma)$. Further in this section we take the box $\mathcal{B} = \mathcal{B}_{\mathfrak{S}} = \{ (\sum_{j} x_{ij} \omega_j)_i \in K_\R^s \colon 0 \leq  x_{ij} < 1  \}$ so that the goal of this section is to study the sum $S_\gamma$ as it was defined in Section~\ref{sec.major_arcs_cubic}. To be completely explicit  with our choice of box we then have
\[
S_\gamma = S(\gamma) = \sum_{0 \leq \mathbf{x} < N(\mathfrak{a}_\gamma)} e \left( \mathrm{tr}(\gamma C(\mathbf{x})) \right),
\]
where the condition $0 \leq \mathbf{x} < N(\mathfrak{a}_\gamma)$ denotes the sum over elements $\mathbf{x} = \left(\sum_{j} x_{ij} \omega_j\right)_i \in \O^s$ such that $0 \leq  x_{ij} < N(\mathfrak{a}_\gamma)$ holds.
The main goal of this section is to establish the bound
\begin{equation} \label{eq.S_gamma_bound}
S_\gamma \ll N(\mathfrak{a}_\gamma)^{s(n-1/6)+\varepsilon}.
\end{equation}
Let $H$ be a positive integer that satisfies $H \leq N(\mathfrak{a}_\gamma)$. Clearly we have
\[
H^{ns} S(\gamma) = \sum_{0 \leq \mathbf{h} < H } \sum_{\substack{ 0 \leq \mathbf{x} < N(\mathfrak{a}_\gamma) \\  0 \leq \mathbf{x}+\mathbf{h} < N(\mathfrak{a}_\gamma)}}e \left( \mathrm{tr}(\gamma C(\mathbf{x}+\mathbf{h})) \right).
\]
Interchanging the order of summation gives
\[
H^{ns} S(\gamma) = \sum_{0 \leq \mathbf{x} < N(\mathfrak{a}_\gamma)} \; \sum_{\substack{ 0 \leq \mathbf{h} < H \\ 0 \leq \mathbf{x} + \mathbf{h} < N(\mathfrak{a}_\gamma)}} e \left( \mathrm{tr}(\gamma C(\mathbf{x}+\mathbf{h})) \right).
\]
Since $H \leq N(\mathfrak{a}_\gamma)$ the number of non-zero summands of the inner sum is bounded by $O(N(\mathfrak{a}_\gamma)^{ns})$. Therefore, an application of Cauchy-Schwarz yields
\[
H^{2ns} |S(\gamma)|^2 \ll N(\mathfrak{a}_\gamma)^{ns} \sum_{0 \leq \mathbf{x} < N(\mathfrak{a}_\gamma)} \; \left\lvert\sum_{\substack{ 0 \leq \mathbf{h} < H \\ 0 \leq \mathbf{x} + \mathbf{h} < N(\mathfrak{a}_\gamma)}} e \left( \mathrm{tr}(\gamma C(\mathbf{x}+\mathbf{h})) \right) \right\rvert^2.
\]
Expanding the square one obtains
\[
H^{2ns} |S(\gamma)|^2 \ll N(\mathfrak{a}_\gamma)^{ns} \sum_{0 \leq \mathbf{x} < N(\mathfrak{a}_\gamma)} \sum_{\substack{ 0 \leq \mathbf{h}_1, \mathbf{h}_2 < H \\ 0 \leq \mathbf{x} + \mathbf{h}_1, \mathbf{x} + \mathbf{h}_2 < N(\mathfrak{a}_\gamma)}} e \left( \mathrm{tr}(\gamma C(\mathbf{x}+\mathbf{h}_1) - C(\mathbf{x}+\mathbf{h}_2)) \right).
\]
Set $\mathbf{y} = \mathbf{x} +\mathbf{h}_2$ and $\mathbf{h} = \mathbf{h}_1-\mathbf{h}_2$. Note that after this change of coordinates  each value of $\mathbf{h}$ in the sum above appears at most $H^{ns}$ times. Therefore the previous display gives
\begin{equation} \label{eq.S_gamma_sum_over_rankythings}
H^{ns} |S(\gamma)|^2 \ll  N(\mathfrak{a}_\gamma)^{ns}  \sum_{|\mathbf{h}| \leq H} \left\lvert T(\mathbf{h}, \gamma)\right\rvert,
\end{equation}
where
\[
T(\mathbf{h}, \gamma) = \sum_{\mathbf{y} \in \mathcal{R}(\mathbf{h})} e\left(\mathrm{tr}( \gamma (C(\mathbf{y}+\mathbf{h}) - C(\mathbf{y}))) \right),
\]
and where $\mathcal{R}(\mathbf{h})$ is a box whose sidelengths are $O(N(\mathfrak{a}_\gamma))$. We take the square of the absolute value of this expression, and expand the resulting sum in order to obtain
\[
| T(\mathbf{h}, \gamma)|^2 = \sum_{\mathbf{y}, \mathbf{z} \in \mathcal{R}(\mathbf{h})} e\left( \mathrm{tr} (\gamma(C(\mathbf{y}+\mathbf{h}) - C(\mathbf{y}) - C(\mathbf{z}+\mathbf{h}) + C(\mathbf{z}) )) \right).
\]
Making the change of variables $\mathbf{y} = \mathbf{z} + \mathbf{w}$ we find
\[
| T(\mathbf{h}, \gamma)|^2 = \sum_{|\mathbf{w}|<N(\mathfrak{a}_{\gamma})}\sum_{\mathbf{z}} e\left( \mathrm{tr} (\gamma C(\mathbf{w}, \mathbf{h}, \mathbf{z})) \right),
\]
where the inner sum ranges over a (potentially empty) box $\mathcal{S}(\mathbf{h},\mathbf{w})$ whose sidelengths are  $O(N(\mathfrak{a}_\gamma))$ and where we write $C(\mathbf{w}, \mathbf{h}, \mathbf{z})$ for the multilinear form given by
\[
C(\mathbf{w}, \mathbf{h}, \mathbf{z}) = C(\mathbf{w}+\mathbf{h}+\mathbf{z}) - C(\mathbf{w}+\mathbf{z}) - C(\mathbf{h}+\mathbf{z}) + C(\mathbf{z}).
\]
In particular we have
\[
C(\mathbf{w}, \mathbf{h}, \mathbf{z}) = 6 \sum_{i=1}^s z_i B_i(\mathbf{w},\mathbf{h}) + \Psi(\mathbf{w},\mathbf{h}),
\]
where $B_i$ are the bilinear forms associated to $C$, and where $\Psi$ is a certain polynomial whose precise shape is of no importance to us. Therefore we find 
\begin{align*}
   | T(\mathbf{h}, \gamma)|^2 = \sum_{\mathbf{w}} \sum_{\mathbf{z}} e\left( \mathrm{tr} \left(6 \gamma \sum_{i=1}^s z_i B_i(\mathbf{w},\mathbf{h}) + \gamma \Psi(\mathbf{w},\mathbf{h})\right) \right).
\end{align*}
Writing $z_i = \sum_j z_{ij} \omega_j$ we may regard the inner sum as an exponential sum over integer variables $z_{ij}$. This is  a linear exponential sum and the coefficient of $z_{ij}$ is given by $6 \mathrm{tr} (\gamma \omega_j B_i(\mathbf{w}, \mathbf{h}))$. A standard argument regarding geometric sums now yields
\[
| T(\mathbf{h}, \gamma)|^2 \ll \sum_{\mathbf{w}} \prod_{i=1}^s \prod_{j=1}^n \min \left(N(\mathfrak{a}_\gamma), \norm{6 \mathrm{tr}( \gamma \omega_j B_i(\mathbf{w}, \mathbf{h})) }^{-1} \right).
\]
In particular the same argument that led to~\eqref{eq.exp_sum_bound_solns_bilinears} shows that
\begin{equation} \label{eq.bound_for_T}
    |T(\mathbf{h}, \gamma)|^2 \ll  N(\mathfrak{a}_\gamma)^{ns+\varepsilon} N(\gamma, N(\mathfrak{a}_\gamma), \mathbf{h}), 
\end{equation}
where
\[
N(\gamma, N(\mathfrak{a}_\gamma), \mathbf{h}) = \# \left\{\mathbf{w} \in \O^s \colon |\mathbf{w}| < N(\mathfrak{a}_\gamma), \, \norm{6 \mathrm{tr}(\gamma \omega_j B_i(\mathbf{w}, \mathbf{h})) } < N(\mathfrak{a}_\gamma)^{-1}, \, \forall i,j \right\}.
\]
Note that the condition in the sum already implies that $6 \Delta B_i(\mathbf{x},\mathbf{y}) \in \mathfrak{a}_\gamma$ holds for all $i$, but we prefer to write it in the above shape in order to highlight the similarities with the argument in the previous section.

As in Section~\ref{sec.weyl_differencing_cubes} we may regard $\mathbf{w} \mapsto \mathrm{tr}(\gamma \omega_j B_i(\mathbf{w}, \mathbf{h}))$ as a linear map $\R^{ns} \rightarrow \R^{ns}$. Hence we can apply Lemma~\ref{lem.shrinking} so that for any $Z \in (0,1]$ we have
\[
N(\gamma, N(\mathfrak{a}_\gamma), \mathbf{h}) \ll Z^{-ns} \# \left\{\mathbf{w} \in \O^s \colon |\mathbf{w}| < Z N(\mathfrak{a}_\gamma), \, \norm{6 \mathrm{tr}(\gamma \omega_j B_i(\mathbf{w}, \mathbf{h})) } < ZN(\mathfrak{a}_\gamma)^{-1}, \, \forall i,j \right\}.
\]
We now wish to choose $Z$ in such a way that we can apply Lemma~\ref{lem.small_torus_norm_implies_divisibility}. In the notation of this lemma we have $m = \Delta \omega_j B_i(\mathbf{w},\mathbf{h})$ and $\theta = 0$. We take $Z \in (0,1]$ such that $Z \asymp  H^{-1} N(\mathfrak{a}_\gamma)^{\frac{1}{n}-1}$ for a suitable implied constant. Then Lemma~\ref{lem.small_torus_norm_implies_divisibility} implies
\[
N(\gamma, P, \mathbf{h}) \ll H^{ns} N(\mathfrak{a}_\gamma)^{ns-s}  \# \left\{\mathbf{w} \in \O^s \colon |\mathbf{w}| < H^{-1} N(\mathfrak{a}_\gamma)^{1/n}, \,   B_i(\mathbf{w}, \mathbf{h})  =0, \, \forall i,j \right\}.
\]
Recalling that $r(\mathbf{h})$ is the rank of $B_i(\mathbf{h}, \cdot) \colon K_\R^s \rightarrow K_\R^s$, using~\eqref{eq.bound_for_T} we find
\[
T(\mathbf{h},\gamma) \ll N(\mathfrak{a}_\gamma)^{ns-\frac{r(\mathbf{h})}{2}+\varepsilon} H^{\frac{nr(\mathbf{h})}{2}}.
\]
Hence~\eqref{eq.S_gamma_sum_over_rankythings} delivers
\[
|S(\gamma)|^2 \ll H^{-ns} N(\mathfrak{a}_\gamma)^{2ns+\varepsilon} \sum_{|\mathbf{h}| \leq H} \left(H^n N(\mathfrak{a}_\gamma)^{-1}\right)^{\frac{r(\mathbf{h})}{2}}.
\]
By~\eqref{eq.rank_condition}, for any $r$ the number of $\mathbf{h}$ with $r(\mathbf{h})=r$ is $O(H^{nr})$. Therefore we find
\[
|S(\gamma)|^2 \ll H^{-ns} N(\mathfrak{a}_\gamma)^{2ns+\varepsilon} \sum_{r=0}^s \left(H^{3n} N(\mathfrak{a}_\gamma)^{-1}\right)^{\frac{r}{2}}.
\]
The sum is maximal either when $r = 0$ or when $r = s$, and thus
\[
|S(\gamma)|^2 \ll H^{-ns} N(\mathfrak{a}_\gamma)^{2ns+\varepsilon} \left(1 + H^{3ns/2} N(\mathfrak{a}_\gamma)^{-s/2} \right).
\]
Choosing $H = \lfloor N(\mathfrak{a}_\gamma) \rfloor^{1/3n}$ this finally yields
\[
S(\gamma) \ll N(\mathfrak{a}_\gamma)^{s(n-1/6)+\varepsilon}.
\]

\subsection{Proof of Theorem~\ref{thm.sing_series_in_13}}
By Theorem~\ref{thm.major_arcs_if_sing_converges} it suffices to show that $\mathfrak{S}(R)$ converges absolutely as $R \rightarrow \infty$.

 Given a positive integer $k$ the number of ideals of $\O$ of norm $k$ is $O(k^\varepsilon)$ using the divisor bound. Hence together with Lemma~\ref{lem.denominator_ideal_fixed}
we obtain that the number of $\gamma \in K/\mathcal{O}$ such that $N(\mathfrak{a}_\gamma) = k$ is bounded by $O(k^{1+\varepsilon})$. Thus, using \eqref{eq.S_gamma_bound} we find
\[
\mathfrak{S}(R) \ll \sum_{k=0}^R k^{-ns+1+\varepsilon} k^{ns-s/6} = \sum_{k=0}^R k^{1-s/6+\varepsilon}.
\]
Therefore $\mathfrak{S}(R)$ converges absolutely to some real number $\mathfrak{S}$ as $R \rightarrow \infty$ provided $s \geq 13$. \qed

We remark that using the ideas of Heath-Brown~\cite[Section 7]{heath2007cubic} it would be possible to establish the absolute convergence of $\mathfrak{S}(R)$ already for $s \ge 11$.

\section{Van der Corput on average} \label{sec.vdc_on_average}

In this section, we work towards a bound for the Weyl sum $S(\alpha)$ on the minor arcs. As observed by Heath-Brown, the simple pointwise van der Corput differencing is not sufficient to improve on Davenport's result for $s \ge 16$.

It is however possible to exploit the fact that we are averaging both over the modulus $\mathfrak{a}_{\gamma}$ as well as the integration variable $\beta$ in the minor arcs, thus leading to a version of van der Corput differencing on average.

From now on we continue to work with the box $\mathcal{B}=\mathcal{B}(\mathbf{z})$ as defined in the beginning of Section~\ref{sec.major_arcs_cubic}. Instead of a pointwise bound for $S(\alpha)$, we will seek to bound the mean-square average
\[
M(\alpha,\kappa)=\int_{\vert \beta-\alpha\vert<\kappa} \vert S(\beta)\vert^2 d\beta
\]
for $\alpha \in K_{\mathbb{R}}$ and a small parameter $\kappa \in (0,1)$, where we remind the reader that the integration is over a region of $K_{\mathbb{R}}$.

In conjunction with the Cauchy-Schwarz inequality and an appropriate dyadic dissection of the minor arcs, a satisfactory bound for $M(\alpha,\kappa)$ will be sufficient to control the total minor arc contribution.

The idea now is that the mean square integral automatically shortens all the $n$ coordinates of $h_1$ in the van der Corput differencing, allowing us to effectively save a factor $\frac{H^n}{P^n}$ over the pointwise bound. Here and throughout we denote $\mathbf{h} = (h_i)_i = \left( \sum_{j}h_{ij} \omega_j \right)_i \in \O^s$.

To this end, we initiate the van der Corput differencing with parameters $1 \le H_{ij} \le P$ to be determined, obtaining
\[
\prod_{i,j} H_{ij} S(\beta)=\sum_{0 \le h_{ij} < H_{ij}} \sum_{\mathbf{x}+\mathbf{h} \in P\mathcal{B}} e\left(\tr(\beta C(\mathbf{x}+\mathbf{h}))\right)=\sum_{\mathbf{x} \in \mathcal{O}^s} \sum_{\mathbf{x}+\mathbf{h} \in P\mathcal{B}} e\left(\tr(\beta C(\mathbf{x}+\mathbf{h}))\right),
\]
where implicitly we still restrict to $\mathbf{h}$ such that $0 \le h_{ij} < H_{ij}$ is satisfied.
Note that the condition $H_{ij} \le P$ ensures that the sum over $\mathbf{x}$ is restricted to $O(P^{ns})$ many summands. An application of Cauchy-Schwarz thus yields
\[
\prod_{i,j} H_{ij}^2 \vert S(\beta)\vert^2 \ll P^{ns}\sum_{\mathbf{x} \in \mathcal{O}^s} \left\vert \sum_{\mathbf{x}+\mathbf{h} \in P\mathcal{B}} e\left(\tr(\beta C(\mathbf{x}+\mathbf{h}))\right)\right\vert^2.
\]
Opening the square on the RHS, we obtain
\[
\prod_{i,j} H_{ij}^2 \vert S(\beta)\vert^2 \ll P^{ns}\sum_{\mathbf{x} \in \mathcal{O}^s} \sum_{\mathbf{x}+\mathbf{h_1}, \mathbf{x}+\mathbf{h_2} \in P\mathcal{B}} e\left(\tr (\beta \left[ C(\mathbf{x}+\mathbf{h_1})-C(\mathbf{x}+\mathbf{h_2})\right])\right).
\]
On substituting $\mathbf{y}=\mathbf{x}+\mathbf{h_2}$ and $\mathbf{h}=\mathbf{h_1}-\mathbf{h_2}$, this becomes
\[
\prod_{i,j} H_{ij}^2 \vert S(\beta)\vert^2 \ll P^{ns}\sum_{\vert h_{ij}\vert \le H_{ij}} w(\mathbf{h}) \sum_{\mathbf{y} \in \mathcal{R}(\mathbf{h})}  e\left(\tr (\beta \left[ C(\mathbf{y}+\mathbf{h})-C(\mathbf{y})\right])\right)
\]
where $w(\mathbf{h})=\#\{\mathbf{h_1},\mathbf{h_2}: \mathbf{h}=\mathbf{h_1}-\mathbf{h_2}\} \le \prod_{i,j} H_{ij}$ and $\mathcal{R}(\mathbf{h})$ is a certain box depending only on $\mathbf{h}$.

Instead of taking absolute values, we now first integrate over $\beta = \sum_{j} \beta_j \omega_j$ with a smooth cutoff function to obtain
\begin{align*}
M(\alpha,\kappa) &\le e^n \int_{K_{\mathbb{R}}} \exp\left(-\frac{\sum_j (\beta_j-\alpha_j)^2}{\kappa^2}\right) \cdot \vert S(\beta)\vert^2 d\beta\\
&\ll \frac{P^{ns}}{\prod_{i,j} H_{ij}^2}\sum_{\vert h_{ij}\vert \le H_{ij}} w(\mathbf{h}) \sum_{\mathbf{y} \in \mathcal{R}(\mathbf{h})} I(\mathbf{h},\mathbf{y})\\
&\ll \frac{P^{ns}}{\prod_{i,j} H_{ij}} \sum_{\vert h_{ij}\vert \le H_{ij}} \left\vert \sum_{\mathbf{y} \in \mathcal{R}(\mathbf{h})} I(\mathbf{h},\mathbf{y})\right\vert,
\end{align*}
where
\begin{align*}
I(\mathbf{h},\mathbf{y})&=\int_{K_{\mathbb{R}}} \exp\left(-\frac{\sum_j (\beta_j-\alpha_j)^2}{\kappa^2}\right) \cdot e\left(\tr (\beta \left[ C(\mathbf{y}+\mathbf{h})-C(\mathbf{y})\right])\right) d\beta\\
&=\pi^{n/2} \kappa^n \prod_{j=1}^n \exp(-\pi^2\kappa^2 \tr (\omega_j\left[ C(\mathbf{y}+\mathbf{h})-C(\mathbf{y})\right])^2) \cdot e(\tr (\alpha\left[ C(\mathbf{y}+\mathbf{h})-C(\mathbf{y})\right])).
\end{align*}
Heuristically, for large $h_1 \in \mathcal{O}$, we should have $C(\mathbf{y}+\mathbf{h})-C(\mathbf{y}) \approx h_1 \cdot \frac{\partial C(\mathbf{y})}{\partial x_1}$ so that by our choice of the box $\mathcal{B}(\mathbf{z})$, this difference is large. But then for some $j$, the trace of this number multiplied with $\omega_j$ must be large as well, leading to a negligible contribution to $M(\alpha,\kappa)$ from those terms, thus effectively cutting down the range to small $h_1$.

We now fix the choice $H_{ij}=H$ for $i \ne 1$ and $H_{1j}=cP$ for a sufficiently small constant $c$ and make the above heuristic discussion precise. For $\mathbf{y} \in \mathcal{R}(\mathbf{h})$ we have
\[C(\mathbf{y}+\mathbf{h})-C(\mathbf{y})=h_1 \cdot \frac{\partial C(\mathbf{y})}{\partial x_1}+O(HP^2+\vert h_1\vert^2 \vert \mathbf{y}\vert).\]
If the width of the box $\mathcal{B}(\bm{z})$ and $c$ are sufficiently small, the fact that $\frac{\partial C(\mathbf{z})}{\partial x_1} \ne 0$ then implies that
\[\left\vert C(\mathbf{y}+\mathbf{h})-C(\mathbf{y})\right\vert \gg \vert h_1\vert \cdot P^2\]
unless $\vert h_1\vert \ll H$. Additionally, unless $\vert h_1\vert \ll \frac{(\log P)^2}{\kappa P^2}$, we even have that
\[\left\vert C(\mathbf{y}+\mathbf{h})-C(\mathbf{y})\right\vert \gg \frac{(\log P)^2}{\kappa}\]
so that for some $j$ we must have
\[\left\vert \tr \left(\omega_j\left[C(\mathbf{y}+\mathbf{h})-C(\mathbf{y})\right]\right)\right\vert \gg \frac{(\log P)^2}{\kappa}\]
and we infer from our previous calculations that the contribution of such $\mathbf{h}$ to $M(\alpha,\kappa)$ is $O(1)$. Hence,
\[M(\alpha,\kappa) \ll 1+\frac{P^{ns-n}}{H^{ns-n}} \sum_{\vert h_i\vert \ll H} \left\vert\sum_{\mathbf{y}} I(\mathbf{h},\mathbf{y})\right\vert\]
if we choose $\kappa \asymp \frac{(\log P)^2}{HP^2}$.

Moreover, the range $\vert \beta-\alpha\vert \ge \kappa \log P$ in the definition of $I(\mathbf{h},\mathbf{y})$ clearly gives a total contribution of $O(1)$ to $M(\alpha,\kappa)$ so that we end up with the estimate
\[
M(\alpha,\kappa) \ll 1+\frac{P^{ns-n}}{H^{ns-n}} \sum_{\vert h_{i}\vert  \ll H} \int_{\vert \beta-\alpha\vert<\kappa \log P} \vert T(\mathbf{h},\beta)\vert d\beta
\]
with
\[
T(\mathbf{h},\beta)=\sum_{\mathbf{y} \in \mathcal{R}(\mathbf{h})}  e\left(\tr (\beta \left[ C(\mathbf{y}+\mathbf{h})-C(\mathbf{y})\right])\right).
\]
As in Section~\ref{sec.vdc_pointwise_sing_series}, we obtain
\[\vert T(\mathbf{h},\beta)\vert^2 \ll P^{ns+\varepsilon} N(\beta,P,\mathbf{h})\]
where
\[N(\beta,P,\mathbf{h})=\#\{\mathbf{w} \in \mathcal{O}^s: \vert \mathbf{w}\vert<P, \|6\tr (\beta\omega_j B_i(\mathbf{w},\mathbf{h}))\|<P^{-1}, \forall i,j\}\]
so that
\begin{equation}\label{eq.malphakappa}
M(\alpha,\kappa) \ll 1+\frac{\kappa^n P^{\frac{3ns}{2}-n+\varepsilon}}{H^{ns-n}} \sum_{\vert h_i\vert \ll H} \max_{\beta \in \mathcal{I}} N(\beta, P,\mathbf{h})^{\frac{1}{2}}
\end{equation}
for $\mathcal{I}=\{\beta: \vert \beta-\alpha\vert \le \kappa \log P\}$.

We next claim that
\[
\max_{\beta \in \mathcal{I}} N(\beta,P,\mathbf{h}) \ll P^{\varepsilon}N(\alpha,P,\mathbf{h})
.\]
Indeed, consider a vector $\mathbf{w}$ counted by $N(\beta,P,\mathbf{h})$. It thus satisfies $\vert \mathbf{w}\vert<P$ as well as $\|6\tr (\beta\omega_j B_i(\mathbf{w},\mathbf{h}))\|<P^{-1}$ so that
\[
\|6\tr (\alpha\omega_j B_i(\mathbf{w},\mathbf{h}))\| \ll \frac{1}{P}+\vert \beta-\alpha\vert \cdot \vert B_i(\mathbf{w},\mathbf{h})\vert \ll \frac{1}{P}+\kappa \log P \cdot HP \ll \frac{(\log P)^3}{P}.
\]
We thus obtain
\[N(\beta,P,\mathbf{h}) \ll \#\{\mathbf{w} \in \mathcal{O}^s: \vert \mathbf{w}\vert<P, \|6\tr (\alpha\omega_j B_i(\mathbf{w},\mathbf{h}))\| \ll \frac{(\log P)^3}{P}, \forall i,j\} \ll P^{\varepsilon}N(\alpha,P,\mathbf{h})\]
where the last estimate is a consequence of Lemma~\ref{lem.shrinking} upon choosing suitable $Z \asymp (\log P)^{-3}$.

We conclude that
\[
M(\alpha,\kappa) \ll 1+\frac{\kappa^n P^{\frac{3ns}{2}-n+\varepsilon}}{H^{ns-n}} \sum_{\vert h_i\vert \ll  H} N(\alpha, P,\mathbf{h})^{\frac{1}{2}}.
\]
Let $\alpha=\gamma+\theta$ with $\gamma \in K$ and $\theta \in K_{\mathbb{R}}$ (which we think of as being small). We are now prepared for an application of Lemmas \ref{lem.shrinking} and \ref{lem.small_torus_norm_implies_divisibility}. Indeed, Lemma \ref{lem.shrinking} implies that
\[
N(\alpha,P,\mathbf{h}) \ll Z^{-ns} \#\{\mathbf{w} \in \mathcal{O}^s: \vert \mathbf{w}\vert<ZP, \|6\tr (\alpha\omega_j B_i(\mathbf{w},\mathbf{h}))\|<ZP^{-1}, \forall i,j\}. 
\]
Following Heath-Brown, we will make two different choices of $Z$: In the first one, we will choose $Z=Z_1$ sufficiently small so that Lemma \ref{lem.small_torus_norm_implies_divisibility} implies that $B_i(\mathbf{w},\mathbf{h})=0$. In the second choice $Z=Z_2$, we will only force $6\Delta B_i(\mathbf{w},\mathbf{h}) \in \mathfrak{a}_{\gamma}$, a consequence followed by a study of how often such a divisibility property can occur, crucially using an average over $\gamma$.

By Lemma \ref{lem.small_torus_norm_implies_divisibility}, if we choose $Z \le 1$ satisfying
\[Z \ll \frac{P}{N(\mathfrak{a}_{\gamma})^{1/n}}\]
and
\[Z \ll \frac{1}{PH\vert \theta\vert N(\mathfrak{a}_{\gamma})^{1/n}}\]
we can conclude that $6\Delta B_i(\mathbf{w},\mathbf{h}) \in \mathfrak{a}_{\gamma}$. If, moreover
\[Z \ll \frac{N(\mathfrak{a}_{\gamma})^{1/n}}{PH}\]
or
\[Z \ll \vert \theta\vert PN(\mathfrak{a}_{\gamma})^{1/n}\]
we obtain that $B_i(\mathbf{w},\mathbf{h})=0$. Here, all the implicit constants need to be sufficiently small in order to satisfy the conditions in Lemma \ref{lem.small_torus_norm_implies_divisibility}.

Writing 
\begin{equation}\label{eq.defeta}
    \eta=\vert \theta\vert+\frac{1}{P^2H}
\end{equation}
we should therefore choose
\[
Z_1 \asymp \min\left(N(\mathfrak{a}_{\gamma})^{1/n}P\eta, \frac{1}{PH\eta N(\mathfrak{a}_{\gamma})^{1/n}}\right),
\]
noting that this automatically implies that $Z_1 \le 1$. Similarly we should choose
\[
Z_2 \asymp \min\left(1, \frac{1}{PH\eta N(\mathfrak{a}_{\gamma})^{1/n}}\right).
\]
In the application with $Z=Z_1$, we thus obtain
\begin{align*}
N(\alpha,P,\mathbf{h}) &\ll Z_1^{-ns} \#\{\mathbf{w} \in \mathcal{O}^s: \vert \mathbf{w}\vert<Z_1P, B_i(\mathbf{w},\mathbf{h})=0, \forall i\}\\
&\ll Z_1^{-ns} \cdot (Z_1P)^{n(s-r)}\\
&\ll P^{ns} \cdot \left(\frac{1}{N(\mathfrak{a}_{\gamma})^{1/n}P^2\eta}+H\eta N(\mathfrak{a}_{\gamma})^{1/n}\right)^{nr}
\end{align*}
with $r=r(\mathbf{h})$.
Instead, in the application with $Z=Z_2$, we end up with the bound
\begin{equation} \label{eq.latticecounting}
N(\alpha,P,\mathbf{h}) \ll Z_2^{-ns} \#\{\mathbf{w} \in \mathcal{O}^s: \vert \mathbf{w}\vert<Z_2P, 6\Delta B_i(\mathbf{w},\mathbf{h}) \in \mathfrak{a}_{\gamma}, \forall i\}.
\end{equation}
We thus need to count vectors $\mathbf{w}$ with $6\Delta B_i(\mathbf{w},\mathbf{h}) \in \mathfrak{a}_{\gamma}$. For any prime ideal $\mathfrak{p}$, let $r_{\mathfrak{p}}(\mathbf{h})$ be the rank of $M(\mathbf{h})$ modulo $\mathfrak{p}$. Clearly, $r_{\mathfrak{p}}(\mathbf{h}) \le r(\mathbf{h})=r$ with strict inequality if and only if $\mathfrak{p}$ divides all $r \times r$ minors of $M(\mathbf{h})$. This means that there are only relatively few such \lq{}bad\rq{} primes, which we will exploit later.

We now decompose $\mathfrak{a}_{\gamma}=\mathfrak{q}_1 \cdot \mathfrak{q}_2$ where $\mathfrak{q}_1$ contains all the primes $\mathfrak{p}$ dividing $\mathfrak{a}_{\gamma}$ with $r_{\mathfrak{p}}(\mathbf{h})<r$ and $\mathfrak{q}_2$ consists of those with $r_{\mathfrak{p}}(\mathbf{h})=r$.

As we are looking for an upper bound, we can replace $\mathfrak{a}_{\gamma}$ by the larger $\mathfrak{q}_2$ in \eqref{eq.latticecounting}.

For fixed $\mathbf{h}$ with $r(\mathbf{h})=r$, the condition $6\Delta B_i(\mathbf{h},\mathbf{w}) \in \mathfrak{q}_2, \forall i$ defines a lattice $\Lambda(\mathbf{h})$ for $\mathbf{w} \in \mathcal{O}^{s}$ which we view as a lattice in $\mathbb{R}^{ns}$.

To estimate the number of integer points in such a lattice we use \cite[Lemma 5.1]{heath2007cubic} implying that
\begin{equation} \label{eq.latticebound}
\#\{\mathbf{x} \in \Lambda(\mathbf{h}): \vert \mathbf{x}\vert \le B \} \ll \prod_i \left(1+\frac{B}{\lambda_i}\right)
\end{equation}
where $\lambda_1,\dots,\lambda_{ns}$ are the successive minima of $\Lambda(\mathbf{h})$.

In order to make this estimate useful, we need a bound on the determinant/covolume $d(\Lambda)$ which is proportional to $\prod_i \lambda_i$ as well as a bound on the skewness of the measure, i.e. upper and lower bounds for the $\lambda_i$.

For the determinant, we note that for $\mathfrak{p}^e \mid \mathfrak{q}_2$, the matrix $M(\mathbf{h})$ has rank $r$ modulo $\mathfrak{p}$ (hence also modulo $\mathfrak{p}^e$) and therefore $B_i(\mathbf{h},\mathbf{w})$ has $N(\mathfrak{p}^e)^{s-r}$ solutions modulo $\mathfrak{p}^e$ so that $N(\mathfrak{p}^e)^r$ divides $d(\Lambda)$.
It thus follows that $N(\mathfrak{q}_2)^r \mid d(\Lambda)$ and hence $d(\Lambda) \ge N(\mathfrak{q}_2)^r$.

Regarding the skewness, we clearly have $\lambda_i \gg 1$ for all $i$, while in the other direction we have $\mathfrak{q}_2\mathcal{O}^s \subset \Lambda(\mathbf{h})$ so that Lemma \ref{lem.elements_in_ideals} implies $\lambda_i \ll N(\mathfrak{q}_2)^{1/n}$.

Optimizing the RHS of \eqref{eq.latticebound} with these constraints shows that the maximum is obtained when $rn$ of the $\lambda_i$ are of order $N(\mathfrak{q}_2)^{1/n}$ while the others are of order $1$.

This shows that
\[
N(\alpha,P,\mathbf{h}) \ll Z_2^{-ns} \left(1+\frac{Z_2P}{N(\mathfrak{q}_2)^{1/n}}\right)^{rn} \cdot (Z_2P)^{(s-r)n}=P^{ns} \left(\frac{1}{Z_2P}+\frac{1}{N(\mathfrak{q}_2)^{1/n}}\right)^{rn}
\]
if $Z_2P \gg 1$ but we note that the bound is trivially true for $Z_2P \ll 1$.

Recalling our choice of $Z_2$, this bound becomes
\[
N(\alpha,P,\mathbf{h}) \ll P^{ns} \left(\frac{1}{P}+\frac{1}{N(\mathfrak{q}_2)^{1/n}}+H\eta N(\mathfrak{a}_{\gamma})^{1/n}\right)^{rn}.
\]
Combining our two estimates, we obtain
\[
N(\alpha,P,\mathbf{h}) \ll P^{ns} \left(\frac{1}{P}+H\eta N(\mathfrak{a}_{\gamma})^{1/n}+\min\left(\frac{1}{N(\mathfrak{a}_{\gamma})^{1/n}P^2\eta},\frac{1}{N(\mathfrak{q}_2)^{1/n}}\right)\right)^{rn}.
\]
We now need to insert this into our expression for $M(\alpha,\kappa)$ which already involves the average over $\mathbf{h}$. Additionally, we want to average over $\mathfrak{a}_{\gamma}$ allowing us to use that $N(\mathfrak{q}_2)$ is almost as large as $N(\mathfrak{a}_{\gamma})$ most of the time.

Our object of study thus becomes
\begin{equation}\label{eq.Atheta}
A(\theta,R,H,P):=\sum_{\gamma: N(\mathfrak{a}_{\gamma})^{1/n} \sim R} \sum_{\vert h_i\vert  \ll H} 
N(\alpha,P,\mathbf{h})^{1/2}
\end{equation}

where we continue to write $\alpha=\gamma+\theta$ and we remind the reader of the notation $q \sim R$ for the dyadic condition $R<q \le 2R$. We then obtain
\[
A(\theta,R,H,P) \ll R^nP^{ns/2} \sum_{\vert h_i\vert \ll H} \sum_{N(\mathfrak{a})^{1/n} \sim R} \left(\frac{1}{P}+H\eta R+\min\left(\frac{1}{RP^2\eta},\frac{1}{N(\mathfrak{q}_2)^{1/n}}\right)\right)^{\frac{r(\mathbf{h})n}{2}}
\]
where we used that there are at most $N(\mathfrak{a})$ choices of $\gamma$ with $\mathfrak{a}_{\gamma}=\mathfrak{a}$ by Lemma~\ref{lem.denominator_ideal_fixed} and we remind the reader that $\mathfrak{q}_2$ depends on $\mathfrak{a}$ and $\mathbf{h}$.

We thus proceed to estimate
\[
V(\mathbf{h},R,\eta):=\sum_{N(\mathfrak{a})^{1/n} \sim R} \min\left(\frac{1}{RP^2\eta},\frac{1}{N(\mathfrak{q}_2)^{1/n}}\right)^{\frac{rn}{2}}
\]
for $r=r(\mathbf{h})$
via a dyadic decomposition as follows:
\begin{align*}
    V(\mathbf{h},R,\eta) &\ll P^{\varepsilon} \max_{S \le R} \sum_{N(\mathfrak{q}_1)^{1/n} \sim S} \sum_{N(\mathfrak{q}_2)^{1/n} \sim \frac{R}{S}}\min\left(\frac{1}{RP^2\eta},\frac{S}{R}\right)^{\frac{rn}{2}}\\
    &\ll P^{\varepsilon} \max_{S \le R} \frac{R^n}{S^n} \min\left(\frac{1}{RP^2\eta},\frac{S}{R}\right)^{\frac{rn}{2}}\#\{\mathfrak{q}_1: N(\mathfrak{q}_1)^{1/n} \le 2S\}.
\end{align*}
Now recall that $\mathfrak{q}_1$ only contains prime ideals dividing a certain non-zero $r \times r$ determinant $M_0$ of $M(\mathbf{h})$. In particular, we have $M_0 \ll H^r$. Applying Rankin's trick, we then obtain
\[
\#\{\mathfrak{q}_1: N(\mathfrak{q}_1)^{1/n} \le 2S\} \ll S^{\varepsilon} \sum_{\mathfrak{q}_1} N(\mathfrak{q}_1)^{-\varepsilon} =S^{\varepsilon} \prod_{\mathfrak{p} \mid M_0} \frac{1}{1-N(\mathfrak{p})^{-\varepsilon}} \ll S^{\varepsilon} M_0^{\varepsilon} \ll P^{\varepsilon}
\]
and thus
\[
V(\mathbf{h},R,\eta) \ll P^{\varepsilon} \max_{S \le R} \frac{R^n}{S^n} \min\left(\frac{1}{RP^2\eta},\frac{S}{R}\right)^{\frac{rn}{2}}.
\]
Maximizing for $S$ we find that 
\[
V(\mathbf{h},R,\eta) \ll P^{\varepsilon} \frac{R^n}{(RP^2\eta)^{rn/2}} \min(1,P^2\eta)^{ne(r)}
\]
with $e(0)=0$, $e(1)=\frac{1}{2}$ and $e(r)=1$ for $r \ge 2$.

Putting everything together, we obtain the estimate
\begin{align*}
A(\theta,R,H,P) &\ll R^{2n}P^{\frac{ns}{2}} \sum_{\vert h_i\vert \ll H} \left[\left(\frac{1}{P}+H\eta R\right)^{\frac{nr(\mathbf{h})}{2}}+\frac{1}{R^n}V(\mathbf{h},R,\eta)\right]\\
&\ll R^{2n}P^{\frac{ns}{2}+\varepsilon} \sum_{\vert h_i\vert \ll H} \left[\left(\frac{1}{P}+H\eta R\right)^{\frac{nr(\mathbf{h})}{2}}+\frac{1}{(RP^2\eta)^{\frac{r(\mathbf{h})n}{2}}} \min(1,P^2\eta)^{ne(r(\mathbf{h}))}\right]\\
&\ll R^{2n}P^{\frac{ns}{2}+\varepsilon} \sum_{r=0}^s H^{nr} \left[\left(\frac{1}{P}+H\eta R\right)^{\frac{nr}{2}}+\frac{1}{(RP^2\eta)^{\frac{rn}{2}}} \min(1,P^2\eta)^{ne(r)}\right]\\
&\ll \left[R^2P^{s/2+\varepsilon} \left(1+(RH^3\eta)^{s/2}+\frac{H^s}{P^{s/2}}+\frac{H^s}{(RP^2\eta)^{s/2}}\min(1,P^2\eta)\right)\right]^n.
\end{align*}
Finally, we argue that the third term $\frac{H^s}{P^{s/2}}$ is negligible.

Indeed, if $HRP\eta \ge 1$, then it is smaller than the second term. Otherwise, if $HRP\eta \le 1$, we have $(RP\eta)^{s/2} \le RP\eta \le \frac{1}{H} \le \min(1,\eta P^2)$ on recalling that $\eta \ge \frac{1}{P^2H}$ and hence the term $\frac{H^s}{P^{s/2}}$ is dominated by the fourth term in that case.

In any case, it now follows that
\begin{equation}\label{eq.Atheta_finalbound}
A(\theta,R,H,P) \ll \left[R^2P^{s/2+\varepsilon} \left(1+(RH^3\eta)^{s/2}+\frac{H^s}{(RP^2\eta)^{s/2}}\min(1,P^2\eta)\right)\right]^n.
\end{equation}

\section{The minor arcs} \label{sec.minor_arcs_cubes}

Finally, we synthesize the bounds obtained by Weyl and van der Corput differencing to estimate the total minor arc contribution $\int_{\mathfrak{m}} S(\alpha) d\alpha$.

We dissect $\mathfrak{m}$ with the help of the version of Dirichlet's Approximation Theorem provided by Lemma~\ref{lem.dir_approx_fractional_quadratic_imag}, applied for some parameter $1 \le Q \le P^{3/2}$ to be determined. Thus, every $\alpha \in K_{\mathbb{R}}$ has an approximation $\alpha=\gamma+\theta$ with $\gamma \in K$ and $N(\mathfrak{a}_{\gamma}) \le Q^n$ as well as $\vert \theta\vert \ll \frac{1}{N(\mathfrak{a}_{\gamma})^{1/n}Q}$.

The assumption $\alpha \in \mathfrak{m}$ then implies that $N(\mathfrak{a}_{\gamma})>P^{\nu}$ or $\vert \theta\vert>P^{-3+\nu}$.
Note that as the contribution to the minor arcs coming from $\vert \theta\vert \le \frac{1}{P^s}$ is $O(Q^{n+1})$, we may assume that $\vert \theta\vert \ge P^{-s}$. 

By a double dyadic decomposition with respect to $\vert \theta\vert$ and $N(\mathfrak{a}_{\gamma})^{1/n}$, we then obtain that
\[
\int_{\mathfrak{m}} S(\alpha) d\alpha \ll Q^{n+1}+P^{\varepsilon}\max_{R \le Q, \phi \le \frac{1}{RQ}} \Sigma(R,\phi)
\]
where
\[
\Sigma(R,\phi):=\sum_{\gamma: N(\mathfrak{a}_{\gamma})^{1/n} \sim R} \int_{\vert \theta\vert \sim \phi} \left\vert S(\gamma+\theta)\right\vert d\theta
\]
and we note that the region of integration is given by a rectangular annulus.

To establish Theorem \ref{thm.asymptotic}, it thus suffices to prove that $\Sigma(R,\phi) \ll P^{n(s-3)-\varepsilon}$. To employ the mean-value estimates from the previous section, we use Cauchy-Schwarz to obtain
\[
\Sigma(R,\phi) \ll R^n\phi^{n/2}\left(\sum_{\gamma: N(\mathfrak{a}_{\gamma})^{1/n} \sim R} \int_{\vert \theta\vert \sim \phi} \left\vert S(\gamma+\theta)\right\vert^2 d\theta\right)^{1/2}.
\]
We next cover the annulus $\vert \theta\vert \sim \phi$ with $O\left(\left(1+\frac{\phi}{\kappa}\right)^n\right)$ boxes of size $\kappa$, all centered at values of $\alpha=\gamma+\theta$ with $\vert \theta\vert \sim \phi$, so that we obtain
\[
\Sigma(R,\phi) \ll R^n\phi^{n/2} \left(1+\frac{\phi}{\kappa}\right)^{n/2} \max_{\vert \theta\vert \sim \phi} \left(\sum_{\gamma: N(\mathfrak{a}_{\gamma})^{1/n} \sim R} M(\gamma+\theta,\kappa)\right)^{1/2}
\]
and using \eqref{eq.malphakappa} and \eqref{eq.Atheta} we obtain
\[
\Sigma(R,\phi) \ll R^n\phi^{n/2} \left(1+\frac{\phi}{\kappa}\right)^{n/2} \max_{\vert \theta\vert \sim \phi} \left(R^{2n+\varepsilon}+\frac{\kappa^n P^{\frac{3ns}{2}-n+\varepsilon}}{H^{ns-n}} A(\theta,R,H,P)\right)^{1/2}
\]
so that \eqref{eq.Atheta_finalbound} implies that
\begin{equation} \label{eq.Sigma_last_bracket}
\Sigma(R,\phi) \ll \left[P^{\varepsilon}R^{2}\phi^{1/2} \left(1+\frac{\phi}{\kappa}\right)^{1/2}\left(1+\frac{\kappa P^{2s-1}}{H^{s-1}} E\right)^{1/2}\right]^n
\end{equation}
where $E=1+(RH^3\eta)^{s/2}+\frac{H^s}{(RP^2\eta)^{s/2}}P^2\eta$. Here we simply estimated $\min(1,P^2\eta) \le P^2\eta$ which turns out to be sufficient.

Suppose we can show that $E \ll 1$. Recall that $\kappa \asymp  \frac{(\log P)^2}{HP^2}$ so that
\[
1+\frac{\phi}{\kappa} \ll \frac{P^{\varepsilon}\eta}{\kappa}
\]
from the definition \eqref{eq.defeta} of $\eta$.

Since $\kappa \gg \frac{1}{P^s}$, both summands in the last bracket of~\eqref{eq.Sigma_last_bracket} are bounded by $\frac{\kappa P^{2s-1}}{H^{s-1}}$. Still assuming $E \ll 1$, we then obtain
\[
\Sigma(R,\phi) \ll \left[P^{\varepsilon}R^{2}\phi^{1/2} \eta^{1/2}\frac{ P^{s-\frac{1}{2}}}{H^{\frac{s-1}{2}}}\right]^n.
\]
Recalling our desired bound $\Sigma(R,\phi) \ll P^{n(s-3)-\varepsilon}$, it now suffices to prove that
\[
H^{s-1} \gg R^4\phi\eta P^{5+\varepsilon},
\]
still under the assumption $E  \ll 1$. Putting $s=14$ for convenience (as we may without loss of generality) and recalling the definition \eqref{eq.defeta} of $\eta$, it suffices to have
\[H^{13} \gg R^4\phi^2P^{5+\varepsilon}\]
as well as
\[H^{14} \gg R^4\phi P^{3+\varepsilon}.\]
We thus choose
\[H \asymp P^{\varepsilon} \max\left\{ \left(R^4\phi^2P^5\right)^{1/13}, \left(R^4\phi P^3\right)^{1/14},1\right\}.\]
In order for this choice to satisfy $H \le P$, we require $R^4\phi^2 \ll P^{8-\varepsilon}$ as well as $R^4\phi \ll P^{11-\varepsilon}$.

Recalling $\phi \le \frac{1}{QR} \le \frac{1}{R^2}$, both conditions are satisfied for any $Q \le P^{3/2}$.

We have thus found an admissible choice of $H$, leading to a satisfactory estimate for $\Sigma(R,\phi)$ under the assumption of $E \ll 1$.

We now enquire under which circumstances this assumption is justified.

For convenience, denote $\phi_0=(R^4P^{31})^{-\frac{1}{15}}$. The relevance of this parameter comes from the observation that for $\phi \le \phi_0$, one has
\[
H \asymp P^{\varepsilon} \max\left\{ \left(R^4\phi P^3\right)^{1/14},1\right\}
\]
and $\eta \asymp \frac{1}{HP^2}$ whereas for $\phi \ge \phi_0$, one has
\[
H \asymp P^{\varepsilon}\max\left\{ \left(R^4\phi^2P^5\right)^{1/13}, 1\right\}
\]
and $\eta \asymp \phi$.

To prove $E \ll 1$, we need to check that $RH^3\eta \ll P^{-\varepsilon}$ as well as $\left(\frac{H^2}{RP^2\eta}\right)^7 P^2\eta \ll P^{-\varepsilon}$.

We begin by checking that $RH^3\eta \ll P^{-\varepsilon}$. First, if $\phi \le \phi_0$, we have
\begin{align*}
    RH^3\eta &\ll P^{\varepsilon} \frac{QH^2}{P^2}\\
    &\ll P^{\varepsilon}\frac{Q}{P^2} \left(1+(R^4\phi P^3)^{1/14}\right)\\
    & \ll P^{\varepsilon} \cdot \left(\frac{Q}{P^2}+\frac{Q^{9/7}}{P^{11/7}}\right).
\end{align*}
This bound is satisfactory if $Q \ll P^{11/9-\varepsilon}$.

Next, if $\phi \ge \phi_0$, we have
\begin{align*}
    RH^3\eta&\ll P^{\varepsilon} RH^3\phi\\
    &\ll P^{\varepsilon} \cdot \frac{1}{Q} \cdot \left(1+\left(R^4\phi^2P^5\right)^{3/13}\right)\\
    &\ll P^{\varepsilon} \cdot \frac{P^{15/13}}{Q}
\end{align*}
which is satisfactory if $Q \gg P^{15/13+\varepsilon}$.

We thus choose $Q=P^{13/11}$, ensuring that $RH^3\eta \ll P^{-\varepsilon}$ in both cases, and noting that this  also satisfies our earlier rough assumption $Q \le P^{3/2}$.

Finally, we need to enquire whether $\left(\frac{H^2}{RP^2\eta}\right)^7 P^2\eta \ll P^{-\varepsilon}$.

For $\phi \le \phi_0$, we have $\eta \asymp \frac{1}{HP^2}$ so that
\[\left(\frac{H^2}{RP^2\eta}\right)^7 P^2\eta \ll P^{\varepsilon} \frac{H^{20}}{R^7}\]
so that it suffices to have $H \ll R^{7/20-\varepsilon}$.

Recalling our choice of $H$ in this case, it is thus sufficient to have $R \gg P^{\varepsilon}$ as well as additionally $\phi \le \phi_1$ where
\[
\phi_1=R^{9/10}P^{-3-\varepsilon}.
\]
Similarly, if $\phi \ge \phi_0$ we have $\eta \asymp \phi$ so that
\[\left(\frac{H^2}{RP^2\eta}\right)^7 P^2\eta \ll \frac{H^{14}}{R^7P^{12}\phi^6}\]
and hence by our definition of $H$, it suffices to have $R\gg P^{\varepsilon}$ as well as additionally $\phi \ge \phi_2$ where
\[\phi_2=\frac{1}{P^{\frac{43}{25}-\varepsilon}R^{7/10}}.\]
Summarizing, we have obtained a satisfactory bound for $\Sigma(R,\phi)$ if $R \gg P^{\varepsilon}$ and $\phi \le \min(\phi_0,\phi_1)$ or $\phi \ge \max(\phi_0,\phi_2)$.

Letting $R_0=P^{4/5+\varepsilon}$, a quick computation shows that $\phi_2 \le \phi_0 \le \phi_1$ if $R \ge R_0$ whereas $P^{-\varepsilon}\phi_1 \le \phi_0 \le \phi_2 P^{\varepsilon}$ if $R \le R_0$.

In the first case, our argument already covers all possible values of $\phi$. We are thus left with the case where $R \le R_0$ and $P^{-\varepsilon}\phi_1 \le \phi \le \phi_2 P^{\varepsilon}$ or $R \le P^{\varepsilon}$.

It is here that we require the bound obtained by Weyl differencing. Indeed, applying Lemma \ref{lem.weyl_estimate} with $s=14$ and noting that the assumption $Q \le P^{3/2}$ is satisfied, we obtain
\[
\Sigma(R,\phi) \ll P^{\varepsilon}\left[R^2\phi P^{14} \left(R\phi+\frac{1}{R\phi P^3}\right)^{7/4}\right]^n.
\]
Recalling our goal $\Sigma(R,\phi) \ll P^{11n-\varepsilon}$, it then suffices to have
\[R^2\phi P^3 \left(R\phi+\frac{1}{R\phi P^3}\right)^{7/4} \ll P^{-\varepsilon}.\]
But this will be satisfied if 
\begin{equation}\label{eq.conditionphi}
\frac{R^{1/3}}{P^{3-\varepsilon}} \ll \phi \ll \frac{1}{P^{12/11+\varepsilon} R^{15/11}}.
\end{equation}
Under the assumption $R \le R_0$ and $P^{-\varepsilon}\phi_1 \le \phi \le \phi_2 P^{\varepsilon}$, this will thus be true as soon as
\[\phi_1 \gg \frac{R^{1/3+\varepsilon}}{P^3}\]
as well as
\[\phi_2 \ll  \frac{1}{P^{12/11+\varepsilon} R^{15/11}}.\]
The first condition is always satisfied for $R \gg P^{\varepsilon}$ while the second one is satisfied for $R \ll P^{\frac{346}{365}-\varepsilon}$ which is indeed true under the assumption $R \le R_0$.

Finally, we need to treat the cases where $R \le P^{\varepsilon}$. Here of course, we need to use that we are on the minor arcs so that $\phi \ge P^{-3+\nu}$. But it is easy to see that in that case \eqref{eq.conditionphi} is also satisfied, thus finishing our proof of Theorem \ref{thm.asymptotic}.\qed

\printbibliography

@article {heath2007cubic,
    AUTHOR = {Heath-Brown, D. R.},
     TITLE = {Cubic forms in 14 variables},
   JOURNAL = {Invent. Math.},
  FJOURNAL = {Inventiones Mathematicae},
    VOLUME = {170},
      YEAR = {2007},
    NUMBER = {1},
     PAGES = {199--230},
      ISSN = {0020-9910},
   MRCLASS = {11E76},
  MRNUMBER = {2336082},
MRREVIEWER = {Timothy D. Browning},
       DOI = {10.1007/s00222-007-0062-1},
       %%url = {https://doi.org/10.1007/s00222-007-0062-1},
}

@article {pleasants1975cubic,
    AUTHOR = {Pleasants, P. A. B.},
     TITLE = {Cubic polynomials over algebraic number fields},
   JOURNAL = {J. Number Theory},
  FJOURNAL = {Journal of Number Theory},
    VOLUME = {7},
      YEAR = {1975},
    NUMBER = {3},
     PAGES = {310--344},
      ISSN = {0022-314X},
   MRCLASS = {10B35 (12A20)},
  MRNUMBER = {453633},
MRREVIEWER = {D. J. Lewis},
       DOI = {10.1016/0022-314X(75)90024-4},
       %%url = {https://doi.org/10.1016/0022-314X(75)90024-4},
}

@article {ryavec1969cubic,
    AUTHOR = {Ryavec, C.},
     TITLE = {Cubic forms over algebraic number fields},
   JOURNAL = {Proc. Cambridge Philos. Soc.},
  FJOURNAL = {Proceedings of the Cambridge Philosophical Society},
    VOLUME = {66},
      YEAR = {1969},
     PAGES = {323--333},
      ISSN = {0008-1981},
   MRCLASS = {10.12},
  MRNUMBER = {244153},
MRREVIEWER = {D. J. Lewis},
       DOI = {10.1017/s0305004100045011},
       %%url = {https://doi.org/10.1017/s0305004100045011},
}

@article {davenport_29_vars,
    AUTHOR = {Davenport, H.},
     TITLE = {Cubic forms in {$29$} variables},
   JOURNAL = {Proc. Roy. Soc. London Ser. A},
  FJOURNAL = {Proceedings of the Royal Society. London. Series A.
              Mathematical, Physical and Engineering Sciences},
    VOLUME = {266},
      YEAR = {1962},
     PAGES = {287--298},
      ISSN = {0962-8444},
   MRCLASS = {10.17 (10.25)},
  MRNUMBER = {136580},
MRREVIEWER = {B. J. Birch},
       DOI = {10.1098/rspa.1962.0062},
       %%url = {https://doi.org/10.1098/rspa.1962.0062},
}

@article {davenport_32_variables,
    AUTHOR = {Davenport, H.},
     TITLE = {Cubic forms in thirty-two variables},
   JOURNAL = {Philos. Trans. Roy. Soc. London Ser. A},
  FJOURNAL = {Philosophical Transactions of the Royal Society of London.
              Series A. Mathematical and Physical Sciences},
    VOLUME = {251},
      YEAR = {1959},
     PAGES = {193--232},
      ISSN = {0080-4614},
   MRCLASS = {10.00},
  MRNUMBER = {105394},
MRREVIEWER = {L. Carlitz},
       DOI = {10.1098/rsta.1959.0002},
       %%url = {https://doi.org/10.1098/rsta.1959.0002},
}

@article {davenport63,
    AUTHOR = {Davenport, H.},
     TITLE = {Cubic forms in sixteen variables},
   JOURNAL = {Proc. Roy. Soc. London Ser. A},
  FJOURNAL = {Proceedings of the Royal Society. London. Series A.
              Mathematical, Physical and Engineering Sciences},
    VOLUME = {272},
      YEAR = {1963},
     PAGES = {285--303},
      ISSN = {0962-8444},
   MRCLASS = {10.17},
  MRNUMBER = {155800},
MRREVIEWER = {B. J. Birch},
       DOI = {10.1098/rspa.1963.0054},
       %%url = {https://doi.org/10.1098/rspa.1963.0054},
}

@article {ramanujam1963cubic,
    AUTHOR = {Ramanujam, C. P.},
     TITLE = {Cubic forms over algebraic number fields},
   JOURNAL = {Proc. Cambridge Philos. Soc.},
  FJOURNAL = {Proceedings of the Cambridge Philosophical Society},
    VOLUME = {59},
      YEAR = {1963},
     PAGES = {683--705},
      ISSN = {0008-1981},
   MRCLASS = {10.17 (10.66)},
  MRNUMBER = {154849},
MRREVIEWER = {D. J. Lewis},
}

@book {davenport_book,
    AUTHOR = {Davenport, H.},
     TITLE = {Analytic methods for {D}iophantine equations and {D}iophantine
              inequalities},
    SERIES = {Cambridge Mathematical Library},
   EDITION = {Second},
    %  NOTE = {With a foreword by R. C. Vaughan, D. R. Heath-Brown and D. E.
    %          Freeman,
     %         Edited and prepared for publication by T. D. Browning},
 PUBLISHER = {Cambridge University Press, Cambridge},
      YEAR = {2005},
     %PAGES = {xx+140},
      ISBN = {0-521-60583-0},
   MRCLASS = {11P05 (11D72 11P55)},
  MRNUMBER = {2152164},
       DOI = {10.1017/CBO9780511542893},
       %%url = {https://doi.org/10.1017/CBO9780511542893},
}

@book{hartshorne2013algebraic,
  title={Algebraic geometry},
  author={Hartshorne, Robin},
  volume={52},
  year={2013},
  publisher={Springer Science \& Business Media}
}

@article {lewis_p_adic_zeroes,
    AUTHOR = {Lewis, D. J.},
     TITLE = {Cubic homogeneous polynomials over {$p$}-adic number fields},
   JOURNAL = {Ann. of Math. (2)},
  FJOURNAL = {Annals of Mathematics. Second Series},
    VOLUME = {56},
      YEAR = {1952},
     PAGES = {473--478},
      ISSN = {0003-486X},
   MRCLASS = {10.0X},
  MRNUMBER = {49947},
MRREVIEWER = {E. R. Kolchin},
       DOI = {10.2307/1969655},
       %%url = {https://doi.org/10.2307/1969655},
}

@article {wooley1997,
    AUTHOR = {Wooley, Trevor D.},
     TITLE = {Linear spaces on cubic hypersurfaces, and pairs of homogeneous
              cubic equations},
   JOURNAL = {Bull. London Math. Soc.},
  FJOURNAL = {The Bulletin of the London Mathematical Society},
    VOLUME = {29},
      YEAR = {1997},
    NUMBER = {5},
     PAGES = {556--562},
      ISSN = {0024-6093},
   MRCLASS = {11D72 (11E76)},
  MRNUMBER = {1458715},
MRREVIEWER = {G. Greaves},
       DOI = {10.1112/S0024609397003184},
       %%url = {https://doi-1org-1qykinsds0130.han.sub.uni-goettingen.de/10.1112/S0024609397003184},
}

@article {birchgoldbach2010,
    AUTHOR = {Br\"{u}dern, J. and Dietmann, R. and Liu, J. Y. and Wooley, T. D.},
     TITLE = {A {B}irch-{G}oldbach theorem},
   JOURNAL = {Arch. Math. (Basel)},
  FJOURNAL = {Archiv der Mathematik},
    VOLUME = {94},
      YEAR = {2010},
    NUMBER = {1},
     PAGES = {53--58},
      ISSN = {0003-889X},
   MRCLASS = {11D72 (11E76 11P32)},
  MRNUMBER = {2581334},
MRREVIEWER = {Scott T. Parsell},
       DOI = {10.1007/s00013-009-0086-4},
       %%url = {https://doi-1org-1qykinsm10010.han.sub.uni-goettingen.de/10.1007/s00013-009-0086-4},
}

@article {dietmannwooley2003,
    AUTHOR = {Dietmann, Rainer and Wooley, Trevor D.},
     TITLE = {Pairs of cubic forms in many variables},
   JOURNAL = {Acta Arith.},
  FJOURNAL = {Acta Arithmetica},
    VOLUME = {110},
      YEAR = {2003},
    NUMBER = {2},
     PAGES = {125--140},
      ISSN = {0065-1036},
   MRCLASS = {11D72 (11E76)},
  MRNUMBER = {2008080},
MRREVIEWER = {Scott T. Parsell},
       DOI = {10.4064/aa110-2-3},
       %%url = {https://doi-1org-1qykinsm1023e.han.sub.uni-goettingen.de/10.4064/aa110-2-3},
}

@article {greentao,
    AUTHOR = {Green, Ben and Tao, Terence},
     TITLE = {The primes contain arbitrarily long arithmetic progressions},
   JOURNAL = {Ann. of Math. (2)},
  FJOURNAL = {Annals of Mathematics. Second Series},
    VOLUME = {167},
      YEAR = {2008},
    NUMBER = {2},
     PAGES = {481--547},
      ISSN = {0003-486X},
   MRCLASS = {11N13 (11A41 11B25 37A45)},
  MRNUMBER = {2415379},
MRREVIEWER = {Tamar Ziegler},
       DOI = {10.4007/annals.2008.167.481},
       %%url = {https://doi-1org-1qykinsm1023e.han.sub.uni-goettingen.de/10.4007/annals.2008.167.481},
}

@article {lewis_57_existence,
    AUTHOR = {Lewis, D. J.},
     TITLE = {Cubic forms over algebraic number fields},
   JOURNAL = {Mathematika},
  FJOURNAL = {Mathematika. A Journal of Pure and Applied Mathematics},
    VOLUME = {4},
      YEAR = {1957},
     PAGES = {97--101},
      ISSN = {0025-5793},
   MRCLASS = {10.00},
  MRNUMBER = {97358},
MRREVIEWER = {G. Whaples},
       DOI = {10.1112/S0025579300001133},
       %%url = {https://doi.org/10.1112/S0025579300001133},
}

@article {birch57existence,
    AUTHOR = {Birch, B. J.},
     TITLE = {Homogeneous forms of odd degree in a large number of
              variables},
   JOURNAL = {Mathematika},
  FJOURNAL = {Mathematika. A Journal of Pure and Applied Mathematics},
    VOLUME = {4},
      YEAR = {1957},
     PAGES = {102--105},
      ISSN = {0025-5793},
   MRCLASS = {10.00},
  MRNUMBER = {97359},
MRREVIEWER = {G. Whaples},
       DOI = {10.1112/S0025579300001145},
       %%url = {https://doi.org/10.1112/S0025579300001145},
}

@article {heath_brown_ten,
    AUTHOR = {Heath-Brown, D. R.},
     TITLE = {Cubic forms in ten variables},
   JOURNAL = {Proc. London Math. Soc. (3)},
  FJOURNAL = {Proceedings of the London Mathematical Society. Third Series},
    VOLUME = {47},
      YEAR = {1983},
    NUMBER = {2},
     PAGES = {225--257},
      ISSN = {0024-6115},
   MRCLASS = {11D72 (11D88 11E76 11G25 11P55)},
  MRNUMBER = {703978},
MRREVIEWER = {D. J. Lewis},
       DOI = {10.1112/plms/s3-47.2.225},
       %url = {https://doi.org/10.1112/plms/s3-47.2.225},
}

@article{Hooley+1988+32+98,
author = {Christopher Hooley},
doi = {doi:10.1515/crll.1988.386.32},
title = {On nonary cubic forms.},
journal = {Journal für die reine und angewandte Mathematik},
number = {386},
%volume = {1988},
year = {1988},
pages = {32--98}
}

@article {browning_vishe_ten,
    AUTHOR = {Browning, T. D. and Vishe, P.},
     TITLE = {Cubic hypersurfaces and a version of the circle method for
              number fields},
   JOURNAL = {Duke Math. J.},
  FJOURNAL = {Duke Mathematical Journal},
    VOLUME = {163},
      YEAR = {2014},
    NUMBER = {10},
     PAGES = {1825--1883},
      ISSN = {0012-7094},
   MRCLASS = {11D72 (11D25 11D45 11P55 11R47 14G05)},
  MRNUMBER = {3229043},
MRREVIEWER = {D. R. Heath-Brown},
       DOI = {10.1215/00127094-2738530},
       %url = {https://doi.org/10.1215/00127094-2738530},
}

@article {skinner_94_thirteen_vars,
    AUTHOR = {Skinner, Christopher M.},
     TITLE = {Rational points on nonsingular cubic hypersurfaces},
   JOURNAL = {Duke Math. J.},
  FJOURNAL = {Duke Mathematical Journal},
    VOLUME = {75},
      YEAR = {1994},
    NUMBER = {2},
     PAGES = {409--466},
      ISSN = {0012-7094},
   MRCLASS = {11P55 (11D72 11R47)},
  MRNUMBER = {1290198},
MRREVIEWER = {D. R. Heath-Brown},
       DOI = {10.1215/S0012-7094-94-07512-1},
       %%url = {https://doi.org/10.1215/S0012-7094-94-07512-1},
}

@article {brandes_dietmann,
    AUTHOR = {Brandes, Julia and Dietmann, Rainer},
     TITLE = {Rational lines on cubic hypersurfaces},
   JOURNAL = {Math. Proc. Cambridge Philos. Soc.},
  FJOURNAL = {Mathematical Proceedings of the Cambridge Philosophical
              Society},
    VOLUME = {171},
      YEAR = {2021},
    NUMBER = {1},
     PAGES = {99--112},
      ISSN = {0305-0041},
   MRCLASS = {11D72 (11D88 11E76 14G05 14J20)},
  MRNUMBER = {4268805},
MRREVIEWER = {Arthur Baragar},
       DOI = {10.1017/S0305004120000079},
       URL = {https://doi.org/10.1017/S0305004120000079},
}

@unpublished{brandes_dietmann_II,
title= {Rational lines on cubic hypersurfaces II},
author = {Brandes, Julia and Dietmann, Rainer},
%year = {12 Oct 2022},
note= {In preparation},
}

@misc{liu2021forms,
      title={On forms in prime variables}, 
      author={Jianya Liu and Lilu Zhao},
      year={2021},
      eprint={2105.12956},
      archivePrefix={arXiv},
      primaryClass={math.NT}
}

@misc{dumke2014quartic,
      title={Quartic Forms in Many Variables}, 
      author={Jan H. Dumke},
      year={2014},
      eprint={1405.7064},
      archivePrefix={arXiv},
      primaryClass={math.NT}
}

@misc{yamagishi_prime,
      title={Diophantine equations in primes: density of prime points on affine hypersurfaces II}, 
      author={Shuntaro Yamagishi},
      year={2021},
      eprint={2111.06122},
      archivePrefix={arXiv},
      primaryClass={math.NT}
}

@article {colliot,
    AUTHOR = {Colliot-Th\'{e}l\`ene, Jean-Louis and Sansuc, Jean-Jacques and
              Swinnerton-Dyer, Peter},
     TITLE = {Intersections of two quadrics and {C}h\^{a}telet surfaces.
              {II}},
   JOURNAL = {J. Reine Angew. Math.},
  FJOURNAL = {Journal f\"{u}r die Reine und Angewandte Mathematik. [Crelle's
              Journal]},
    VOLUME = {374},
      YEAR = {1987},
     PAGES = {72--168},
      ISSN = {0075-4102,1435-5345},
   MRCLASS = {11G35 (11E81 14J20 14J26)},
  MRNUMBER = {876222},
MRREVIEWER = {Noriko\ Yui},
}

@article {brandes_note_p_adic,
    AUTHOR = {Brandes, Julia},
     TITLE = {A note on {$p$}-adic solubility for forms in many variables},
   JOURNAL = {Bull. Lond. Math. Soc.},
  FJOURNAL = {Bulletin of the London Mathematical Society},
    VOLUME = {47},
      YEAR = {2015},
    NUMBER = {3},
     PAGES = {501--508},
      ISSN = {0024-6093,1469-2120},
   MRCLASS = {11D72 (11E76 11P55)},
  MRNUMBER = {3354445},
MRREVIEWER = {Nikos\ Tzanakis},
       DOI = {10.1112/blms/bdv023},
       URL = {https://doi.org/10.1112/blms/bdv023},
}

\end{document}